\documentclass{lmcs} %
\pdfoutput=1

\usepackage{lastpage}
\lmcsdoi{20}{1}{19}
\lmcsheading{}{\pageref{LastPage}}{}{}%
{Mar.~07,~2023}{Mar.~07,~2024}{}


\usepackage{hyperref}
\usepackage{color}
\usepackage{xspace}
\usepackage{proof}
\usepackage{url}
\usepackage{amsthm, amssymb}
\usepackage{xifthen}
\usepackage{xargs}
\usepackage{needspace}
\usepackage[countsame]{coqtheorem}

\hypersetup{
    breaklinks=true,
}

\setBaseUrl{{https://uds-psl.github.io/coq-library-fol/tennenbaum/}}

\newcoqtheorem{definition}{Definition}
\newcoqtheorem{lemma}{Lemma}
\newcoqtheorem{theorem}{Theorem}
\newcoqtheorem{corollary}{Corollary}
\newcoqtheorem{observation}{Observation}
\newcoqtheorem{coqfact}{Fact}
\newcoqtheorem{coqrem}{Remark}
\newcoqtheorem{coqaxiom}{Axiom}
\newcoqtheorem{coqhyp}{Hypothesis}

\usepackage[utf8]{inputenc}
\usepackage{mathpartir}
\usepackage{arydshln}
\usepackage{color}
\usepackage{xspace}
\usepackage{proof}
\usepackage{url}
\usepackage{multicol}
\usepackage{amsthm}

\newcommand{\MBB}[1]{\ensuremath{\mathbb{#1}}\xspace}  %
\newcommand{\MCL}[1]{\ensuremath{\mathcal{#1}}\xspace} %
\newcommand{\MSF}[1]{\ensuremath{\mathsf{#1}}\xspace}  %

\newcommand{\Nat}{\MBB{N}}   %
\newcommand{\Bool}{\MBB{B}}  %
\newcommand{\Unit}{\mymathbb{1}}  %

\renewcommand{\iff}{\mathbin{\leftrightarrow}}
\newcommand{\Iff}{\mathbin{\Longleftrightarrow}}

\newcommand{\col}{\!:\!}

\newcommand{\incl}{\ensuremath{\subseteq}}

\newcommand{\cdef}{\mathbin{:=}}
\newcommand{\bnfdef}{\mathbin{::=}}

\DeclareMathAlphabet{\mymathbb}{U}{bbold}{m}{n}
\newcommand{\fun}{\lambda}

\newcommand{\Prop}{\MBB{P}}

\newcommand{\btrue}{\MSF{tt}}
\newcommand{\bfalse}{\MSF{ff}}

\newcommand{\Void}{\mymathbb{0}}
\newcommand{\SigType}[2]{\Sigma{#1}.\,{#2}}

\newcommand{\cnil}{\mathalpha{[\,]}}
\newcommand{\ccons}{\mathbin{::}}
\newcommand{\Opt}[1]{\MCL{O}{(#1)}}
\newcommand{\List}[1]{\MSF{List}{(#1)}}
\newcommand{\Vector}[2]{{#2}^{#1}}
\newcommand{\norm}[1]{\vert{#1} \vert}
\newcommand{\some}[1]{{}^\circ{#1}}
\newcommand{\none}{\emptyset}

\newcommand{\definite}[1]{\MSF{definite} \, {#1}}

\newcommand{\stable}[1]{\MSF{stable} \, {#1}}
\newcommand{\stablePred}[1]{\MSF{stable} \, {#1}}

\newcommand{\M}{\MCL{M}}

\newcommand{\Q}{\MSF{Q}}
\newcommand{\PA}{\MSF{PA}}
\newcommand{\HA}{\MSF{HA}}
\newcommand{\CTT}{\MSF{CTT}}
\newcommand{\IZF}{\MSF{IZF}}
\newcommand{\ZF}{\MSF{ZF}}
\newcommand{\WS}{\MSF{WS}}
\newcommand{\CIC}{\MSF{CIC}}

\newcommand{\CT}{\MSF{CT}}
\newcommand{\WCT}{\MSF{WCT}}
\newcommand{\CTQ}{\MSF{CT_Q}}
\newcommand{\WCTQ}{\MSF{WCT_Q}}

\newcommand{\MP}{\MSF{MP}}
\newcommand{\LEM}{\MSF{LEM}}
\newcommand{\DNE}{\MSF{DNE}}
\newcommand{\UC}{\MSF{AUC}}

\newcommand{\dec}[1]{\MSF{dec} {#1}}
\newcommand{\decider}[2]{\MSF{decider}\,{#1}\,{#2}}
\newcommand{\Dec}[1]{\MSF{Dec} {#1}}

\newcommand{\primeFunc}[1]{\pi_{#1}}

\newcommand{\FuncSignature}{\MCL{F}}
\newcommand{\PredSignature}{\MCL{P}}

\newcommand{\Term}{\MSF{tm}}
\renewcommand{\Form}{\MSF{fm}}

\newcommand{\cvdash}{\vdash_c}

\newcommand{\enumform}[1]{\Phi_{#1}}

\newcommand{\szero}{0}
\newcommand{\natS}{S}
\newcommand{\num}[1]{\overline{#1}}
\newcommand{\sequal}{=}
\newcommand{\sless}{<}
\newcommand{\sle}{\leq}
\newcommand{\sneg}{\neg}
\newcommand{\simp}{\to}
\newcommand{\siff}{\leftrightarrow}
\newcommand{\sand}{\land}
\newcommand{\sor}{\lor}
\newcommand{\sforall}{\forall \,}
\newcommand{\sexists}{\exists \,}

\newcommand{\interpFunc}[2]{{#1}^{#2}}
\newcommand{\interpPred}[2]{{#1}^{#2}}
\newcommand{\evalTerm}[2]{ \hat{#1} \, {#2} }
\newcommand{\evalForm}[3]{{#1} \vDash_{#2} {#3} }
\newcommand{\envcons}{;}
\newcommand{\satformula}[2]{ {#1} \vDash {#2} (\num{\, \cdot \,}) }

\newcommand{\Inum}[1]{\nu \, {#1}}
\newcommand{\InumAbrv}[1]{\num{#1}}
\newcommand{\Izero}[1]{\szero^{#1}}
\newcommand{\Isucc}[2]{\natS^{#1}{#2}}

\newcommand{\standardElement}[1]{\MSF{std}{#1}}
\newcommand{\divNat}[1]{ \, \num{\, \cdot \,} \mid {#1} }

\renewcommand{\iff}{\leftrightarrow}

\def\cf{{\em cf.}}

\newcommand{\Bibkeyhack}[3]{}

\begin{document}

\title[An Analysis of Tennenbaum's Theorem in Constructive Type Theory]{An Analysis of Tennenbaum's Theorem\texorpdfstring{\\}{} in Constructive Type Theory\rsuper*}
\titlecomment{{\lsuper*}This is an extended version of a paper published at FSCD 2022~\cite{hermes2022tennenbaum}.}

\author[M.~Hermes]{Marc Hermes\lmcsorcid{0000-0002-0375-759X}}[a,b]
\author[D.~Kirst]{Dominik Kirst\lmcsorcid{0000-0003-4126-6975}}[b]

\address{Radboud University Nijmegen, Netherlands}	%
\email{Marc.Hermes@cs.ru.nl}  %

\address{Universität des Saarlandes, Saarbrücken, Germany}	%
\email{kirst@ps.uni-saarland.de}  %

\begin{abstract}
\noindent Tennenbaum's theorem states that the only countable model of Peano arithmetic (PA) with computable arithmetical operations is the standard model of natural numbers.
In this paper, we use constructive type theory as a framework to revisit, analyze and generalize this result.

The chosen framework allows for a synthetic approach to computability theory, exploiting that, externally, all functions definable in constructive type theory can be shown computable.
We then build on this viewpoint, and furthermore internalize it by assuming a version of Church's thesis, which expresses that any function on natural numbers is representable by a formula in PA\@.
This assumption provides for a conveniently abstract setup to carry out rigorous computability arguments, even in the theorem's mechanization.

Concretely, we constructivize several classical proofs and present one inherently constructive rendering of Tennenbaum's theorem, all following arguments from the literature.
Concerning the classical proofs in particular, the constructive setting allows us to highlight differences in their assumptions and conclusions which are not visible classically.
All versions are accompanied by a unified mechanization in the Coq proof assistant.
\end{abstract}

\maketitle

\section{Introduction}\label{Introduction}

There are well-known proofs in classical logic showing that the first-order theory of Peano arithmetic (\PA) has non-standard models, meaning models which are not isomorphic to the standard model \Nat{}.
One way to do so (\cf~\cite{boolos2002computability}), starts by adding a new constant symbol \( c \) to the language of \PA{}, together with the enumerable list of new axioms \( c \neq 0 \),  \( c \neq 1 \), \( c \neq 2 \) etc.
This theory has the property that every finite subset of its axioms is satisfied by the standard model \Nat, since we can always give a large enough interpretation of the constant \( c \) in \Nat.
Hence, by the compactness theorem, the full theory has a model \M, which
must then be non-standard since the interpretation of \( c \) in \M corresponds to an element which is larger than any number \( n \) in \( \Nat{} \).

The presence of non-standard elements like this has interesting consequences.
\PA{} can prove that for every bound \( n \), sums of the form \( \sum_{k \leq n} a_k \) exist, so in particular for example the Gaussian sum \( \sum_{k \leq n} k \). 
The presence of the non-standard element \( c \) in \( \M \) allows for the creation of infinite sums like \( \sum_{k \leq c} k \). 
Remarkably, while this means that it includes a summation over all natural numbers, the model specifies a single element as the result of this infinite sum.
A general \PA{} model \( \M \) therefore exhibits behaviors which disagree with the common intuition that computations in \PA{} are finitary, which are---in the end---largely based on the familiarity with the standard model \Nat.

\enlargethispage{\baselineskip}

These intuitions are still not too far off the mark, as was demonstrated by Stanley Tennenbaum~\cite{tennenbaum1959non} in a remarkable theorem.
By being a little more restrictive on the models under consideration, \Nat{} regains a unique position:
\begin{quote}
  \textbf{Tennenbaum's Theorem:}
  Apart from the standard model \Nat, there is no countable \emph{computable} model of first-order \PA{}\@.
\end{quote}
A model is considered \emph{computable} if its elements can be coded by numbers in \Nat{}, and the arithmetic operations on its elements can be realized by computable functions on these codes.
Usually, Tennenbaum's theorem is formulated in a classical framework such as \ZF{} set theory, and the precise meaning of \emph{computable} is given by making reference to a concrete model of computation like Turing machines, \( \mu \)-recursive functions, or the \( \fun \)-calculus~\cite{kaye2011tennenbaum, smith2014tennenbaum}.
In a case like this, where computability theory is applied rather than developed, the computability of a function is rarely proven by exhibiting an explicit construction in the specific model, but rather by invoking the informal \emph{Church-Turing thesis}, stating that every function intuitively computable is computable in the chosen model.
Proving the computability of a function is then reduced to giving an intuitive justification. 

The focus of this paper lies on revisiting Tennenbaum's theorem and several of its proofs in a constructive type theory (\CTT).
In contrast to classical treatments, the usage of a constructive meta-theory enables us to formally assume \emph{Church's thesis}~\cite{kreisel1970church,troelstra1988constructivism,forster:LIPIcs:2021:13455} in the form of an axiom, stating that every total function is computable.
By its usage, the elegant and succinct paper-style computability proofs can be reproduced, but in a fully formal manner, also allowing for straightforward mechanized proofs.

In the constructive type theory that we will specify in Section~\ref{section_preliminaries}, the addition of this axiom becomes possible since we can adapt the approach of \emph{synthetic computability}~\cite{richman1983church,bauer2006first,forster2019synthetic}:
Any function term that is definable in \CTT{} by the virtue of its syntactic rules, can externally be observed to be a computable function.
Following through on this external observation, it can be taken as a justification to also internally treat functions as if they were computable. 
For example, we will make use of this perspective when defining a predicate on a type \( X \) to be \emph{decidable} if there exists a function \( f \col X \to \Bool \) computing booleans which reflect the truth values of \(p\) (\autoref{def_decidability}).

This approach leads to a simplification when it comes to the statement of Tennenbaum's theorem itself:
By interpreting arithmetic operations as type-theoretic functions in \CTT{}, all models are automatically computable in the synthetic sense.
Consequently, ``computable model'' no longer needs to be part of the theorem statement. 
Moreover, the given synthetic proofs unveil the computational essence of Tennenbaum's theorem with neither the technical overhead of constructions in a formal model of computability nor the informal practice to simply disregard computability arguments.
The paper therefore contributes to an active line of research on mechanized metamathematics~\cite{KirstThesis} with the two-fold goal to obtain a uniform and constructive development of the foundations of mathematics, complemented by a principled mechanization for guaranteed correctness and collaborative use.

In the above sketched framework, we follow the classical presentations of Tennenbaum's theorem~\cite{kaye2011tennenbaum, smith2014tennenbaum} and develop constructive versions that only assume a type-theoretic version of \emph{Markov's principle}~\cite{mannaa2017independence}.
This is then complemented by the adaption of an inherently more constructive variant given by McCarty~\cite{mccarty1987variations,McCarty}.

Concretely, our contributions can be summarized as follows:
\begin{itemize}
	\item	We review several existing proofs of Tennenbaum's theorem drawn from the literature, and carry them out in a constructive meta-theory. 
  We work out subtle differences in the strengths of their conclusions, which are left invisible in any classical treatment, but come to light once they are viewed through a constructive lens.
	\item	By considering models with a decidable divisibility relation (\autoref{tennenbaum-inseparable}), we extend the theorem to models which do not have to be discrete or enumerable.
	\item	We provide a Coq mechanization covering all results presented in this paper.\footnote{
    The full mechanization is available on
    \textcolor{ACMDarkBlue}{\href{https://github.com/uds-psl/coq-library-fol/tree/tennenbaum}{GitHub}}~\cite{hermes2023repo}
    and can conveniently be viewed on a \textcolor{ACMDarkBlue}{\href{https://uds-psl.github.io/coq-library-fol/tennenbaum/toc.html}{webpage}}~\cite{kirst2023website}.
    We make use of a fact for which we gave no mechanized proof, namely Hypothesis~\ref{hypothesis_object_level_coding} in Section~\ref{section_Variants}.
  }
\end{itemize}
The present paper is an extended version of~\cite{hermes2022tennenbaum} and adds the following contributions:
\begin{itemize}
  \item In~\cite{hermes2022tennenbaum}, we only gave a reference to a possible proof strategy for showing the existence of \HA{}-inseparable formulas (\autoref{def_HA_inseparable}). 
  A comment in~\cite{peters2022goedel} pointed to a more straightforward approach, which we were able to mechanize, and have added to Section~\ref{section_Variants}.
  \item We added a short discussion in Section~\ref{section_root_of_Tennenbaum}, in which we aim to abstractly identify the main ingredients that are used in Section~\ref{section_Variants} to derive Tennenbaum's theorem.
  \item Our Coq mechanization has been re-based and now relies on a collaborative Coq library for first-order logic~\cite{kirst2022library}, to which we have contributed by integrating the mechanization developed for this publication. This integration not only enabled the utilization of pre-existing definitions for \( \Delta_1 \) and \(\Sigma_1 \)--formulas but also provides additional validation for our application of Church's thesis for Robinson arithmetic (\CTQ{}), given 
  the library includes its derivation from a more conventional formulation of Church's thesis~\cite{kirst2023goedel}.
  \item The presentation of several proofs, definitions and theorem statements in Section~\ref{section_coding} and Section~\ref{section_Tennenbaum} has been revised. Additionally, a mistake in~\cite{hermes2022tennenbaum} has been corrected; the original version of \HA-coding (\cf~\cite{hermes2022tennenbaum} Hypothesis 2) was not constructively provable, the new one (\autoref{hypothesis_object_level_coding}) is.
\end{itemize}
To conclude this introduction, we give a brief overview on the structure of the paper, in the order that we consider most suitable for a first reading:

The main results of the analysis are summarized in Section~\ref{section_discussion}, where we give a tabulated overview on the different variants of Tennenbaum's theorem that result from the various proofs.
It clarifies which assumptions are made for each version, and we give a brief discussion of what to take from these differences.
The complete proofs are covered in Section~\ref{section_Tennenbaum}, ending with Section~\ref{section_root_of_Tennenbaum} in which we abstractly capture the essence of Tennenbaum's theorem.

In Section~\ref{section_CT} we motivate and introduce our chosen formulation of Church's thesis which is utilized as an axiom. 
Basic results about \PA{}'s standard and non-standard models are shown in Section~\ref{section_HA_models} and then used in Section~\ref{section_coding} to establish results that allow the encoding of predicates on \Nat{}, which are essential in the proof of Tennenbaum's theorem.

To make the paper self-contained, we also give an introduction to the essential features of constructive type theory, synthetic computability, and the type-theoretic specification of first-order logic in Section~\ref{section_preliminaries}. This is continued in Section~\ref{secion_PA_axiomatization} by the presentation of the first-order axiomatization of \PA{} as given in previous work~\cite{kirst2021synthetic,kirst2023synthetic}.

\section{Preliminaries}\label{section_preliminaries}

\subsection{Constructive Type Theory}\label{section_CTT}

The chosen framework for this paper is a constructive type theory (\CTT).
More specifically, it will be the calculus of inductive constructions (\CIC)~\cite{coquand1986calculus, paulin1993inductive} which is implemented in the Coq proof assistant~\cite{the_coq_development_team_2023_8161141}. 
It provides a predicative hierarchy of \emph{type universes} above a single impredicative universe \( \Prop \) of \emph{propositions} and the capability of inductive type definitions.
On the type level, we have the unit type \( \Unit \) with a single element,
the void type \( \Void \), function spaces \( X\to Y \), products \( X \times Y \), sums \( X + Y \), dependent products\footnote{
  As is custom in Coq, we write \( \forall \) in place of the symbol \( \Pi \) for dependent products.
}
\( \forall (x:X).\,A\,x \), and dependent sums \( \Sigma (x:X).\,A\,x \).
On the propositional level, analogous notions to the one listed for types in the above are present, but denoted by their usual logical notation (\( \top \), \( \bot \), \( \to \), \( \land \), \( \lor \), \( \forall \), \( \exists \)).\footnote{
  Negation \( \neg A \) is used as an abbreviation for both \( A \to \bot \) and \( A \to \Void \).
}
It is important to note that the so-called \emph{large eliminations} from the impredicative \( \Prop \) into higher types of the hierarchy are restricted.
In particular, it is generally not possible to show \( (\exists x. \, p \, x) \to \Sigma x. \, p \, x \).\footnote{
  The direction \( (\Sigma x. \, p \, x) \to \exists x. \, p \, x \) is however always provable.
  Intuitively, one can think of \( \exists x. \, p \, x \) as stating the mere existence of a value satisfying \( p \), while \( \Sigma x. \, p \, x \) is a type that also carries a value satisfying \(p\).
}
The restriction does however allow for large elimination of the equality predicate \( = \) of type \( \forall X.\,X\to X\to\Prop \), as well as function definitions by well-founded recursion.

We will also use the basic inductive types
of \emph{Booleans} (\( \Bool:= \btrue\mid\bfalse \)),
\emph{Peano natural numbers} (\( n:\Nat:= 0\mid n+1 \)),
the \emph{option type} (\( \Opt X:= \some{x}\mid\none \)) and
\emph{lists} (\( l:\List X:=\cnil\mid x\ccons l \)).
Furthermore, by \( X^n \) we denote the type of \emph{vectors} \( \vec v \) of length \( n:\Nat{} \) over \( X \).

Given predicates \( P, Q \col X \to \Prop \) on a type \( X \), we will occasionally use the set notation \( P \incl Q \) for expressing \( \forall x \col X. \, P\, x \to Q \, x \).
\setCoqFilename{FOL.Tennenbaum.SyntheticInType}
\begin{definition}[ ][definite]\label{classical_principles}
  A proposition \( P \col \Prop \) is called \coqlink[definite]{\emph{definite}} if \( P \lor \neg P \) holds and \coqlink[stable]{\emph{stable}} if \( \neg \neg P \to P \).
  The same terminology is used for predicates \( p \col X \to \Prop \) given they are \coqlink[Definite]{pointwise definite} or \coqlink[Stable]{stable}.
  We furthermore want to recall the following logical principles:
  \begin{alignat*}{2}
    \coqlink[LEM]{\LEM} &\cdef \forall P \col \Prop. \, \definite{P}
    \hspace{1.5em} &\text{(Law of Excluded Middle)}
    \\
    \coqlink[DNE]{\DNE} &\cdef \forall P \col \Prop. \, \stable{P}
    \hspace{1.5em} &\text{(Double-Negation Elimination)}
    \\
    \coqlink[MP]{\MP} &\cdef \forall f \col \Nat{} \to \Nat. \, \stable{(\exists n. \, f n = 0)}
    \hspace{1.5em} &\text{(Markov's Principle)}
  \end{alignat*}
\end{definition}
\noindent
Note that \coqlink[LEM_DNE]{\LEM{} and \DNE{} are equivalent} and not provable in \CIC{}, while \MP{} is much weaker and has a constructive interpretation~\cite{mannaa2017independence}.
For convenience, and as used by Bauer~\cite{bauer2008potentially}, we adapt the reading of double negated statements like \( \neg \neg P \) as ``\emph{potentially \( P \)}''.\footnote{
  A referee pointed out to us that \( \neg \neg P \) is closely related to \( \Box \Diamond P \) in modal logic, where \( \Box \) stands for \emph{necessity} and \( \Diamond \) for \emph{possibility}. The statement can therefore also be read as ``by necessity, \(P\) must be possible''.
}
\setCoqFilename{FOL.Tennenbaum.DN_Utils}
\begin{coqrem}[\coqlink{Handling \( \neg \neg \)}]\label{DN_handling}
  Given any propositions \( A, B \) we constructively have the equivalence {\( (A \to \neg B) \iff (\neg \neg A \to \neg B) \)}, meaning that when trying to prove a negated goal, we can remove double negations in front of any assumption.
  More generally, any statement of the form
  \(
    \neg \neg A_1 \to \hdots \to \neg \neg A_n \to \neg \neg C
  \)
  is equivalent to
  \(
    A_1 \to \hdots \to A_n \to \neg \neg C
  \)
  and since {\( C \to \neg \neg C \)} holds, it furthermore suffices to show \( A_1 \to \hdots \to A_n \to C \) in this case.
  In the following, we will make use of these facts without further notice.
\end{coqrem}

\subsection{Synthetic Computability}\label{section_synthetic}

As already expressed in Section~\ref{Introduction}, constructive type theory permits us to take a viewpoint that considers all functions to be computable functions, yielding simple definitions~\cite{forster2019synthetic} of many textbook notions of computability theory:
\setCoqFilename{FOL.Tennenbaum.SyntheticInType}
\begin{definition}[Enumerability][enumerable]
  Let \( p : X \to \Prop \) be some predicate.
  We say that \( p \) is \coqlink[enumerable]{\emph{enumerable}} if there is an \emph{enumerator} \( f : \Nat{} \to \Opt X \) such that \( \forall x \col X. \, p \, x \iff \exists n. \, f n = \some x \).
\end{definition}

\begin{definition}[Decidability][Dec]\label{def_decidability}
  Let \( p : X \to \Prop \) be some predicate.
  We call \( f : X \to \Bool \) a \coqlink[decider]{\emph{decider}} for \( p \) and write \( \decider{p}{f} \) iff \( \forall x \col X. \, p \, x \iff f x = \btrue \).
  We then define the following notions of decidability:
  \begin{itemize}
    \item \coqlink[Dec_decider]{{\( \Dec{\, p} \)} \( \cdef ~ \exists f \col X \to \Bool. \, \decider{p}{f} \)}
    \item \coqlink[dec]{{\( \dec{( P \col \Prop)} \)} \( \cdef ~ P + \neg P \)}.
  \end{itemize}
  In both cases we will often refer to the predicate or proposition simply as being {\emph{decidable}}.
\end{definition}
We also expand the synthetic vocabulary with notions for types.
In the textbook setting, many of them can only be defined for sets which are in bijection with \Nat, but synthetically they can be handled in a very uniform way.
\begin{definition}[ ][Enumerable]\label{type_properties}
  We call a type \( X \)
  \begin{itemize}
    \item \coqlink[enumerable_equiv]{\emph{enumerable}} if \( \fun x \col X. \top \) is enumerable,
    \item \coqlink[Discrete_equiv]{\emph{discrete}} if there exists a decider for equality \( = \) on \( X \),
    \item \coqlink[Separated_equiv]{\emph{separated}} if there exists a decider for apartness \( \neq \) on \( X \),
    \item \coqlink[Witnessing_equiv]{\emph{witnessing}} if \( \forall f \col X \to \Bool. \, (\exists x. \, f x = \btrue) \to \SigType{x}{f x = \btrue} \).
  \end{itemize}
\end{definition}
\setCoqFilename{FOL.Tennenbaum.MoreDecidabilityFacts}
\begin{coqfact}[ ][Witnessing_nat]\label{Nat_witnessing}
  In \CIC, the types \coqlink[Witnessing_nat]{\Nat{}} and \coqlink[Witnessing_natnat]{\( \Nat^2 \)} are witnessing. \qed{}
\end{coqfact}

\subsection{First-Order Logic}\label{section_FOL}

In order to study Tennenbaum's theorem, we need to give a description of the first-order theory of \PA{} and the associated intuitionistic theory of \emph{Heyting arithmetic} (\HA), which has the same axiomatization, but uses intuitionistic first-order logic.
We follow prior work in~\cite{forster2019synthetic,forster2021completeness,kirst2021synthetic} and describe first-order logic as embedded inside the constructive type theory, by inductively defining formulas, terms, and the deduction system.
We then define a semantics for this logic, which uses Tarski models and interprets formulas over the respective domain of the model.
The type of natural numbers \Nat{} will then naturally be a model of \HA{}.

Before specializing to one particular theory, we keep the definition of first-order logic general and fix some arbitrary signature \( \Sigma = (\FuncSignature; \PredSignature) \) for function and predicate symbols.
\setCoqFilename{undec/Undecidability.FOL.Syntax.Core}
\begin{definition}[Terms and Formulas][term]
  We define \coqlink[term]{terms} {\( t \col \Term \)} and \coqlink[form]{formulas} {\( \varphi \col \Form \)} inductively:
  \begin{align*}
    s, t  \col \Term &\bnfdef
      x_n \mid f \, \vec v
      \hspace{0.6cm} (n \col \Nat, ~f \col \MCL F, ~\vec v \col \Vector{\norm{f}}{\Term})
    \\
    \alpha, \beta \col \Form &\bnfdef
      \bot \mid
      P \, \vec v \mid
      \alpha \simp \beta \mid
      \alpha \sand \beta \mid
      \alpha \sor \beta \mid
      \sforall \alpha \mid
      \sexists \beta
      \hspace{0.6cm} (P \col \MCL P, ~\vec v \col \Vector{\norm{P}}{\Term}),
  \end{align*}
  where \( \vert f \vert \) and \( \vert P \vert \) are the arities of the function symbol \( f \) and predicate symbol \( P \), respectively.
\end{definition}
We use de Bruijn indexing to formalize the binding of variables to quantifiers. This means that the variable \( x_n \) at some position in a formula is \emph{bound} to the \( n \)-th quantifier preceding this variable in the syntax tree of the formula.
If there is no quantifier binding the variable, it is said to be \emph{free}.
\begin{definition}[Substitution][subst_term]
  Given a variable assignment \( \sigma : \Nat \to \Term \) we recursively define \coqlink[subst_term]{\emph{substitution} on terms} by \( x_k [\sigma] \cdef \sigma \, k \) and \( f \, \vec v \cdef f (\vec v [\sigma]) \), and \coqlink[subst_form]{extended to formulas} by
  \begin{equation*}
    \bot [\sigma] \cdef \bot
    \hspace{1.4em}
    (P \, \vec v)[\sigma] \cdef P \, (\vec v [\sigma])
    \hspace{1.4em}
    (\alpha ~\dot \Box~ \beta) [\sigma] \cdef \alpha[\sigma] ~\dot \Box~ \beta[\sigma]
    \hspace{1.3em}
    (\dot \nabla \, \varphi)[\sigma] \cdef \dot \nabla (\varphi [x_0 \envcons \fun x. (\sigma x)[\uparrow]])
  \end{equation*}
  where \( \dot \Box \) is any logical connective and \( \dot \nabla \) any quantifier.
  The variable assignment \coqlink[scons]{\( x \envcons \sigma \)} is defined by
  \( (x \envcons \sigma) \, 0 \cdef x \) as well as
  \( (x \envcons \sigma) (n+1) \cdef \sigma \, n \)
  and simply appends \( x \) as the first element to \( \sigma \col \Nat{} \to \Term \).
  By \( \uparrow \) we designate the assignment \( \lambda n. \, x_{n+1} \) shifting all variable indices by one.
\end{definition}

\setCoqFilename{undec/Undecidability.FOL.Deduction.FullND}
\begin{definition}[Natural Deduction][ ]\label{def_natural_deduction}
Natural deduction \coqlink[prv]{\( \vdash ~: \List{\Form} \rightarrow \Form \rightarrow \Prop \)} is characterized inductively by the usual rules (see \autoref{appendix_natural_deduction}).
We write \( \vdash \) for intuitionistic natural deduction and \( \cvdash \) for the classical variant, which extends \( \vdash \) by adding every instance of Peirce's law \( ((\varphi \to \psi) \to \varphi) \to \varphi \).\footnote{
  Another way to treat the distinction between classical and intuitionistic theories would be to add all instances of Peirce's law to the axioms of a theory, instead of building them into the deduction system.
}
\end{definition}

\setCoqFilename{undec/Undecidability.FOL.Semantics.Tarski.FullCore}
\begin{definition}[Tarski Semantics][ ]\label{def_Tarski_semantics}
  A \emph{model} \( \M \) consists of a type \( D \) designating its domain together with functions \( \interpFunc{f}{\M} : D^{\norm{f}} \to D \) and \( \interpPred{P}{\M} : D^{\norm{P}} \to \Prop \) for all symbols \( f \) in \( \FuncSignature \) and \( P \) in \( \PredSignature \).
  We will also use \( \M \) to refer to the domain.
  Functions \( \rho : \Nat{} \to \M \) are called environments and are used as variable assignments to recursively give \coqlink[eval]{evaluations to terms}:
  \begin{align*}
    \evalTerm{\rho}{x_k} \cdef \rho \, k
    \hspace{3em}
    \evalTerm{\rho}{(f \, \vec v)} \cdef \interpFunc{f}{\M} (\evalTerm{\rho}{\vec v})
    \hspace{2.5em}
    (v \col \Vector{n}{\Term})
  \end{align*}
  This interpretation is then extended to formulas via \coqlink[sat]{the satisfaction relation}:
  \begin{alignat*}{2}
    \evalForm{\M}{\rho}{P \, \vec v} ~&\cdef~ \interpPred{P}{\M} (\evalTerm{\rho}{\vec v})
    &
    \evalForm{\M}{\rho}{\alpha \simp \beta} ~&\cdef~ \evalForm{\M}{\rho}{\alpha} \to \evalForm{\M}{\rho}{\beta}
    \\
    \evalForm{\M}{\rho}{\alpha \sand \beta} ~&\cdef~ \evalForm{\M}{\rho}{\alpha} \land \evalForm{\M}{\rho}{\beta}
    &\hspace{3em}
    \evalForm{\M}{\rho}{\alpha \sor \beta} ~&\cdef~ \evalForm{\M}{\rho}{\alpha} \lor \evalForm{\M}{\rho}{\beta}
    \\
    \evalForm{\M}{\rho}{\sforall \alpha} ~&\cdef~ \forall x \col D.~ \evalForm{\M}{x \envcons \rho}{\alpha}
    &
    \evalForm{\M}{\rho}{\sexists \alpha} ~&\cdef~ \exists x \col D.~ \evalForm{\M}{x \envcons \rho}{\alpha}
  \end{alignat*}
  We say that a formula \( \varphi \) \emph{holds in the model} \( \M \) and write \( \M \vDash \varphi \) if for every \( \rho \) we have \( \evalForm{\M}{\rho}{\varphi} \). We extend this notation to theories \( \MCL T \col \Form \to \Prop \) by writing \( \M \vDash \MCL T \) iff \( \forall \varphi. \, \MCL T \, \varphi \rightarrow \M \vDash \varphi \), and we write \( \MCL T \vDash \varphi \) if \( \M\vDash \varphi \) for all models \( \M \) with \( \M \vDash \MCL T \).
  Given a \( e \col \M \), we also use the notation \( \M \vDash \varphi(e) \) for all \( \rho \) we have \( \M \vDash_{e ;\rho} \varphi \).
\end{definition}

\setCoqFilename{undec/Undecidability.FOL.Semantics.Tarski.FullSoundness}
\begin{coqfact}[Soundness][tsoundness]
  For any formula \( \varphi \) and theory \( \MCL T \), if \( \MCL T \vdash \varphi \) then \( \MCL T \vDash \varphi \). \qed{}
\end{coqfact}
From the next section on, we will use conventional notation with named variables instead of explicitly writing formulas with de Bruijn indices.

\section{Axiomatization of Peano Arithmetic and Heyting Arithmetic}\label{secion_PA_axiomatization}

\setCoqFilename{undec/Undecidability.FOL.Arithmetics.FA}

We present \PA{} following~\cite{kirst2021synthetic}, as a first-order theory with a signature consisting of symbols for the constant zero, the successor function, addition, multiplication and equality:
\[
  {\Sigma_\PA{} \cdef (\FuncSignature_\PA{} ; \PredSignature_\PA) = \textstyle (\szero , \, \natS , +  ,  \times  ;  \sequal )}
\]
The finite core of \PA{} axioms consists of statements characterizing the successor function, as well as addition and multiplication:
\begin{align*}
  \setCoqFilename{undec/Undecidability.FOL.Arithmetics.PA}
  \coqlink[PAeq_discr]{\text{Disjointness: }} & {\sforall} x. \, S x \sequal \szero \simp \bot  &
  \setCoqFilename{undec/Undecidability.FOL.Arithmetics.PA}
  \coqlink[PAeq_inj]{\text{Injectivity: }} & {\sforall}  x y . \,  S x \sequal S y \simp x \sequal y \\
  \setCoqFilename{undec/Undecidability.FOL.Arithmetics.FA}
  \coqlink[ax_add_zero]{+ \text{-base: }} & {\sforall}  x. \, \szero  +  x \sequal x &
  \coqlink[ax_add_rec]{+ \text{-recursion: }} & {\sforall}  x y . \, (\natS x)  +  y \sequal S (x  +  y)  \\
  \coqlink[ax_mult_zero]{\times\text{-base: }} & {\sforall}  x. \, \szero \times x \sequal \szero  &
  \coqlink[ax_mult_rec]{\times\text{-recursion: }} & {\sforall} x y . \, (\natS x) \times y \sequal y  +  x \times y
\end{align*}
We then get the full (and infinite) axiomatization of {\PA} with the axiom scheme of induction for unary formulas.
In our meta-theory the schema is a type-theoretic function on formulas:
\setCoqFilename{undec/Undecidability.FOL.Arithmetics.PA}
\[
  \coqlink[ax_induction]{\lambda \varphi.\, \varphi[\szero] \simp ( \sforall x.\, \varphi[x] \simp \varphi[S x] ) \simp \sforall x.\, \varphi[x]}
\]
If instead of the induction scheme we add the axiom {\( \sforall x. \, x = \szero \sor \sexists y. \, x = S y \)}, we get the theory {\Q} known as \emph{Robinson arithmetic}.
Both \PA{} and \Q also contain axioms for equality:
\setCoqFilename{undec/Undecidability.FOL.Arithmetics.FA}
\begin{align*}
  \coqlink[ax_refl]{\text{Reflexivity: }} & {\sforall}  x . \,  x \sequal x \\
  \coqlink[ax_sym]{\text{Symmetry: }} & {\sforall}  x y. \,  x \sequal y \simp y \sequal x \\
  \coqlink[ax_trans]{\text{Transitivity: }} & {\sforall}  x y z. \,
  x \sequal y \simp y \sequal z \simp x \sequal z \\
  \coqlink[ax_succ_congr]{\natS \text{-equality: }} & {\sforall}  x y. \,
  x \sequal y \simp \natS x \sequal \natS y \\
  \coqlink[ax_add_congr]{ + \text{-equality: }} & {\sforall}  x y u v. \,
  x \sequal u \simp y \sequal v \simp  x  +  y \sequal u  +  v  \\
  \coqlink[ax_mult_congr]{\times\text{-equality: }} & {\sforall}  x y u v. \,
  x \sequal u \simp y \sequal v \simp  x \times y \sequal u \times v
\end{align*}
Using classical derivability \( \vdash_c \, \) we get the classical first-order theory of Peano arithmetic \( \PA \vdash_c \). 
Its intuitionistic counterpart \( \PA \vdash \) uses intuitionistic derivability \( \vdash \) and is called \emph{Heyting arithmetic}.
Given that the constructive type theory selected for this work only provides a model for the intuitionistic theory, we will restrict our focus on Heyting arithmetic. 
To emphasize this, we will from now on write \HA{} instead of \PA{}.

For simplicity, we only consider models that interpret the equality symbol with the actual equality relation of its domain, so-called \emph{extensional} models.
In the Coq development we make the equality symbol a syntactic primitive, therefore enabling the convenient behavior that the interpreted equality reduces to actual equality.
\setCoqFilename{undec/Undecidability.FOL.Arithmetics.FA}
\begin{definition}[ ][num]
  We recursively define a function \coqlink[num]{\( \num{\, \cdot \,} : \Nat{} \rightarrow \Term \)} by \( \num{0} \cdef \szero \) and \( \num{n+1} \cdef \natS{ \num{n} } \), giving every natural number a representation as a term. Any term \( t \) which is of the form \( \num{n} \) will be called \emph{numeral}.
\end{definition}
We furthermore use notations for expressing \emph{less than} \( x \sless y \cdef \sexists k. \, \natS (x  +  k) \sequal y \), \emph{less than or equal} \( x \sle y \cdef \sexists k. \, x  +  k \sequal y \) and for \emph{divisibility} \( x \mid y \cdef \sexists k. \, x \times k \sequal y \).

The formulas of \HA{} can be classified in a hierarchy based on their computational properties.
We will only consider two levels of this hierarchy:
\setCoqFilename{FOL.Incompleteness.qdec}
\begin{definition}[\(\Delta_1\) and \(\Sigma_1\)--formulas (\cf~\cite{kirst2023goedel})][Qdec]\label{def_DeltaForm}
  A formula \( \varphi \) is 
  \coqlink[Qdec]{\( \Delta_1 \)} if for every assignment \( \sigma \) that only substitutes closed terms, we have \( \Q \vdash \varphi[\sigma] \) or \( \Q \vdash \neg \varphi[\sigma] \).
  \setCoqFilename{FOL.Incompleteness.sigma1}
  A formula is \coqlink[53c9961cd487ec928874d3989eac19a2]{\( \Sigma_1 \)}
  if it is of the form \( \exists x_1 \dots \exists x_n. \, \varphi_0 \), where \( \varphi_0 \)  is a \( \Delta_1 \)--formula.
\end{definition}

Given a \( \Sigma_1 \)--formula \( \exists x_1 \hdots \exists x_n. \varphi \) where \( \varphi \) is \( \Delta_1 \), we can prove it equivalent to the formula \( \exists x \, \exists x_1 < x \hdots \exists x_n < x. \, \varphi \).
Since \(\Delta_1\)--formulas stay \(\Delta_1\) if they are closed by a bounded quantifier, this shows that the initial \(\Sigma_1 \)--formula can be written as a \( \Delta_1 \)--formula, which is preceded by exactly one existential quantifier. We will occasionally make use of this fact and refer to it as 
\setCoqFilename{FOL.Incompleteness.sigma1}
\coqlink[8eb9c83d193eda350750450f09b3dfb3]{\emph{\( \Sigma_1 \)--compression}}.
A more syntactic definition of \( \Delta_1 \) would characterize them as the formulas which are equivalent to both a \( \Pi_1 \) and \( \Sigma_1 \)--formula.
For our purposes the definition which only stipulates the necessary decidability properties is sufficient, as it implies the absoluteness and completeness properties we will need~\cite{kirst2023goedel}:
\setCoqFilename{FOL.Incompleteness.sigma1}
\begin{coqfact}[\( \Sigma_1 \)--Completeness][4a8f1009ad08f25a76601bb3b9104242]\label{Sigma1_completeness}
  For any \( \Sigma_1 \)--formula \( \varphi \) we have \( \Nat \vDash \varphi \) iff \( \HA \vdash \varphi \). \qed{}
\end{coqfact}
\setCoqFilename{FOL.Tennenbaum.Tennenbaum_insep}
\begin{coqfact}[\( \Delta_1 \)--Absoluteness][Delta1_absoluteness]\label{Delta1_absoluteness}
  Let \( \M \vDash \HA \) and \( \varphi \) be any closed \( \Delta_1 \)--formula, then we have \( \Nat \vDash \varphi \to \M \vDash \varphi \). \qed{}
\end{coqfact}

\section{Standard and Non-standard Models of HA}\label{section_HA_models}

From now on \( \M \) will always designate a \HA{} model. Any model like this has an interpretation \(\Izero{\M}\) of the zero symbol, as well as an interpretation \( \Isucc{\M}{} \col \M \to \M \) of the symbol for the successor. By repeated application of \( \Isucc{\M}{} \) we can therefore get the sequence of elements \( \Izero{\M}, \Isucc{\M}{}\Izero{\M} , \Isucc{\M}{}\Isucc{\M}{}\Izero{\M}, \hdots  \), essentially giving us a copy of the standard numbers inside \M.
We will now put this intuition more formally.

\setCoqFilename{FOL.Tennenbaum.Peano}
\begin{coqfact}[ ][inu]\label{def_Inum}\label{eval_num}
  We recursively define a function \coqlink[inu]{\( \Inum{} : \Nat{} \to \M \)} by \( \Inum{0} \cdef \Izero{\M} \) and \( \Inum{(n+1)} \cdef \Isucc{\M}{(\Inum{\, n})} \).
  Furthermore, we define the predicate \coqlink[std]{\( {\standardElement{}} \cdef \fun e. \, \exists n. \,\Inum{n} = e \)} and refer to \( e \) as a \emph{standard number} if \( \standardElement{ \, e} \) and \emph{non-standard} if \( \neg \, \standardElement{\, e} \).
  We then have
  \begin{enumerate}
    \coqitem[eval_num]{\( \evalTerm{\rho}{\num{n}} = \Inum{n} \) for any \( n \col \Nat{} \) and environment \( \rho \col \Nat \to \M \).}
    \coqitem[inu_inj]{\( \Inum{} \) is an injective homomorphism and therefore an \emph{embedding} of \Nat{} into \M.}
  \end{enumerate}
  Both facts are taken as justification to abuse notation and also write \( \num{n} \) for \( \Inum{n} \). \qed{}
\end{coqfact}
Usually we would have to write \( \Izero{\M}, \Isucc{\M}{},  + ^\M, \times^\M, \sequal^\M \) for the interpretations of the respective symbols in a model \( \M \).
For better readability we will however take the freedom to overload the symbols \( 0, S, +, \cdot, = \) to also refer to these interpretations.
\begin{definition}[ ][stdModel']
  \M is called a \emph{standard model} if there is a bijective homomorphism \( \varphi : \Nat \to \M \).
  We will accordingly write \( \M \cong \Nat \) if this is the case.
\end{definition}
We can show that \( \Inum{} \) is essentially the only homomorphism from \( \Nat \) to \M we need to worry about, since it is unique up to extensional equality of functions:
\begin{lemma}[ ][hom_agree_inu]\label{unique_nat_hom}
  Let \( \varphi : \Nat \to \M \) be a homomorphism, then \( \forall x \col \Nat. \, \varphi \, x = \Inum{x} \).
\end{lemma}
\begin{proof}
  By induction on \( x \) and using the fact that both are homomorphisms.
\end{proof}
We now have two equivalent ways to express standardness of a model.
\begin{lemma}[ ][stdModel_eqiv]\label{std_model_equiv_definitions}
  \( \M \cong \Nat \) iff \(\forall e \col \M. \, \standardElement{\, e}\).
\end{lemma}
\begin{proof}
  Given \( \M \cong \Nat \), there is an isomorphism \( \varphi : \Nat \to \M \). Since \( \varphi \) is surjective, \autoref{unique_nat_hom} implies that \( \Inum{} \) must also be surjective.
  For the converse: if \(\Inum{}\) is surjective, it is an isomorphism since it is injective by \autoref{def_Inum}.
\end{proof}
Having seen that every model contains a unique embedding of \( \Nat \), one may wonder whether there is a formula \( \varphi \) which could define and pick out precisely the standard numbers in \( \M \).
\autoref{definability_std} gives a negative answer to this question:

\setCoqFilename{FOL.Tennenbaum.Coding}
\begin{lemma}[ ][predicate_equiv]\label{definability_std}
  There is a unary formula \(\varphi(x)\) with
  \(
    \forall e \col \M. \, \big( \, \standardElement{\, e} \, \leftrightarrow \, \M \vDash \varphi(e) \, \big)
  \) if and only if \( \M \cong \Nat{} \).
\end{lemma}
\begin{proof}
  Given a formula \( \varphi \) with the stated property, we certainly have \( \M \vDash \varphi(\num{0}) \) since \( \InumAbrv{0} \) is a standard number, and clearly \( \M \vDash \varphi(x) \implies \standardElement{\, x} \implies \standardElement{\, (S x)} \implies \M \vDash \varphi(S x) \).
  Thus, by induction in the model, we have \( \M \vDash  \forall x. \, \varphi(x) \), which is equivalent to \( \forall e \col \M. \, \standardElement{\, e} \).
  The converse implication holds by choosing the formula \( x = x \).
\end{proof}
We now turn our attention to models which are not isomorphic to \Nat.

\setCoqFilename{FOL.Tennenbaum.Coding}
\begin{coqfact}[ ][num_lt_nonStd]\label{Std_lt_nonStd}
  For any \( e\col\M \), we have \( \neg \, \standardElement{\, e} \) iff \( \forall n \col \Nat. ~ e > \num{n} \). \qed{}
\end{coqfact}
\setCoqFilename{FOL.Tennenbaum.Peano}
\begin{definition}[ ][nonStd]\label{def_nonStdModel}
  Founded on the result of \autoref{Std_lt_nonStd} we write \( e > \Nat \) iff \( \neg \, \standardElement{\, e} \) and call \( \M \)
  \begin{itemize}
    \item \coqlink[nonStd]{\emph{non-standard}} (written \( \M > \Nat{} \)) iff there is \( e \col \M \) such that \( e > \Nat \),
    \item \coqlink[notStd]{\emph{not standard}} (written \( \M \not \cong \Nat{} \)) iff \( \neg \M \cong \Nat{} \).
  \end{itemize}
  We will also write \( e \col \M > \Nat \) to express that  \( e \) is a non-standard element in \( \M \).
\end{definition}
Of course, we have \coqlink[nonStd_notStd]{\( \M > \Nat \to \M \not \cong \Nat \)}, but the converse implication does not hold constructively in general, so the distinction of both notions becomes meaningful.
\setCoqFilename{FOL.Tennenbaum.Coding}
\begin{lemma}[Overspill][Model.Overspill]\label{Overspill}
  If \( \M \not \cong \Nat \) and \( \varphi(x) \) is unary with \( \M \vDash \varphi(\InumAbrv{n}) \) for every \( n \col \Nat \), then
  \begin{enumerate}
    \coqitem[Overspill]{\( \neg \, \big ( \forall e \col \M. \, \M \vDash \varphi(e) \to \standardElement{\, e} \big ) \)}
    \coqitem[Overspill_DN]{\( \stablePred{\, \standardElement{}} \to \neg \neg \, \exists \, e > \Nat. \, \M \vDash \varphi(e) \)}
    \coqitem[DNE_Overspill]{\DNE{} \(  \to \exists \, e > \Nat. \, \M \vDash \varphi(e). \)}
  \end{enumerate}
\end{lemma}
\begin{proof}
  \( (1) \) Assuming \( \forall e \col \M. ~  \M \vDash \varphi(e) \to \standardElement{\, e} \) and combining it with our assumption that \( \varphi \) holds on all numerals, \autoref{definability_std} implies \( \M \cong \Nat{} \), giving us a contradiction.
  For \( (2) \) note that we constructively have that \( \neg \exists e \col \M. \, \neg \standardElement{\, e} \,\land \, \M \vDash \varphi(e) \) implies \( \forall e \col \M. \, \M \vDash \varphi(e) \to \neg \neg \, \standardElement{\, e} \),
  and by using the stability of \( \standardElement{} \) we therefore get a contradiction in the same way as in \( (1) \).
  Statement \( (3) \) immediately follows from \( (2) \).
\end{proof}
From \autoref{Overspill} we learn that under certain conditions, whenever a formula is satisfied on all standard numbers \( \num n \), this satisfaction ``spills over'' into the non-standard part of the model, meaning there is a non-standard element which also satisfies the formula.
In the next section, we will encounter our first application of this principle.

\section{Coding Finite and Infinite Predicates}\label{section_coding}

There is a standard way in which finite sets of natural numbers can be encoded by a single natural number.
Assuming we have some injective function \(\primeFunc{} : \Nat{} \to \Nat{} \) whose image consists only of prime numbers, and given a finite set of numbers like \( S \cdef \{ 4, 13, 21, 33 \} \), we can encode this set by the single number \( c \cdef \pi_4 \cdot \pi_{13} \cdot \pi_{21} \cdot \pi_{33} \).
It then satisfies \( n \in S \iff \pi_n \mid c \), allowing us to reconstruct \( S \) by checking which primes divide \(c\).

Instead of applying this to sets, we can also use it to encode bounded portions of predicates on \Nat. 
\setCoqFilename{FOL.Tennenbaum.Coding}
\begin{lemma}[ ][Coding_nat']\label{bounded_coding_nat}
  {Given \( n \col \Nat \) and any predicate \( p : \Nat \rightarrow \Prop \) with \( \forall x < n. \, p \, x \lor \neg \, p \, x \), we have}
  \[
    \exists c \col \Nat{} ~ \forall u \col \Nat.
    \big(u < n \rightarrow (p \, u \leftrightarrow \primeFunc{u} \mid c) \big) \land \big(\primeFunc{u} \mid c \rightarrow u < n \big)
  \]
  The right part of the conjunction assures that no primes above \( \primeFunc{n} \) end up in the code \(c\).
\end{lemma}
\begin{proof}
  We do a proof by induction on \( n \).
  For \( n = 0 \) we can choose \( c \cdef 1 \).
  In the induction step, the induction hypothesis gives us a code \( c \col \Nat{} \) which codes \( p \) up to \( n \).
  Since by assumption, \( p \) is definite below \( S n \), we know that \( p\,n \lor \neg p\,n \), allowing us to consider two cases:
  If  \( p \, n \), we set the new code to be \( c' \cdef c \cdot \primeFunc{n} \), if \( \neg p \, n \) we simply set \( c' \cdef c \).
  In both cases one can now verify that \( c' \) will correctly code \( p \) up to \(S n\).
\end{proof}

\begin{corollary}[Finite Coding in \Nat][Coding_nat]\label{coding_nat}
  {Given any \( p : \Nat \rightarrow \Prop \) and bound \( n \col \mathbb{N} \), we have}
  \[
    \neg \neg \, \exists c \col \Nat{} ~ \forall u \col \Nat.
    \big(u < n \rightarrow (p \, u \leftrightarrow \primeFunc{u} \mid c) \big) \land \big(\primeFunc{u} \mid c \rightarrow u < n \big)
  \]
  Note that {if \( p \) is definite}, we can drop the \( \neg \neg \).
\end{corollary}
\begin{proof}
  If \( p \) is definite, we trivially have \( \forall x < n. \, p \, x \lor \neg \, p \, x \), so \autoref{bounded_coding_nat} gives us the \( \neg \neg \)--free existence as claimed.
  Without assuming definiteness, we can still constructively show \( \neg \neg (\forall x < n. \, p \, x \lor \neg \, p \, x) \) by induction on \( n \), which combined with \autoref{bounded_coding_nat} gives us the existence, but behind a double negation.
\end{proof}
With a proof of the encoding in \Nat{} we can give a straightforward proof that this is possible in any model of \HA{}.
\begin{coqrem}[ ]\label{remark_prime_formula}
  To formulate the above result in a generic model \( \M \vDash \HA{} \), we require an object-level representation of the prime function \( \primeFunc{} \).
  For now, we will simply assume that we have such a binary formula \( \Pi(x,y) \) and defer the justification to Section~\ref{section_CT}.
\end{coqrem}
The statement ``\( \primeFunc{u} \) divides \( c \)'' can now be expressed by \( \sexists p . ~\Pi (u, p) \sand p \mid c \), for which we will abuse notation and simply write \( \Pi(u) \mid c \).
\begin{lemma}[Finite Coding in \M][Coding_model_binary]\label{finite_coding_model}
  {For any binary formula \( \alpha(x,y) \) and \( n \col \Nat{} \) we have}
  \begin{align*}
    \M \vDash \sforall e \, \neg \neg \, \sexists c \, \sforall u \sless \num{n}.
    ~\alpha(u,e) \siff \Pi(u) \mid c.
  \end{align*}
\end{lemma}
\begin{proof}
  Let \( e \col \M \), and define the predicate \( p \cdef \lambda u \col \Nat. \, \M \vDash \alpha(\num{u}, e) \).
  Then \autoref{coding_nat} potentially gives us a code \( a \col \Nat{} \) for \( p \) up to the bound \( n \).
  It now suffices to show that the actual existence of \( a \col \Nat{} \) already implies
  \begin{align*}
    \M \vDash \sexists c \, \sforall u \sless \num{n}.
    ~ \alpha(u,e) \siff \Pi(u) \mid c.
  \end{align*}
  And indeed, we can verify that \( c = \InumAbrv{a} \) shows the existential claim:
  given \( u \col \M \) with \( \M \vDash u \sless \num{n} \) we can conclude that \( u \) must be a standard number \( \InumAbrv{u} \).
  We then have the equivalences
  \begin{align*}
    \M \vDash \alpha(\num{u},e)
    ~\Iff~
    p \, u
    ~\Iff~
    \primeFunc{u} \mid a
    ~\Iff~
   \M \vDash \Pi (\num{u}) \mid \num{a}
  \end{align*}
  since \( a \) codes \( p \) and \( \Pi \) represents \( \primeFunc{} \).
\end{proof}
Overspill now has interesting consequences when it comes to encoding, as for models that are not standard, it allows for the potential encoding of a complete predicate \( p : \Nat{} \to \Prop \), and therefore also of infinite subsets.
\begin{lemma}[Infinite Coding in \M][Coding_nonStd_unary]\label{coding_model_unary}
  {If \( \standardElement{} \) is stable, \( \M \not \cong \Nat{} \) and \( \alpha(x) \) unary, we have}
  \[
    \neg \neg \, \exists c \col \M ~ \forall u \col \Nat.
    ~~ \M \vDash
    \alpha(\num{u}) \siff
    \Pi (\num{u}) \mid c.
  \]
\end{lemma}
\begin{proof}
  Using \autoref{finite_coding_model} for the present case where \( \alpha \) is unary, we get
  \begin{align*}
    \M \vDash \neg \neg \, \sexists c \, \sforall u \sless \num n.
    ~ \alpha(u) \siff
    \Pi (u) \mid c
  \end{align*}
  for every \( n \col \Nat{} \), so by \autoref{Overspill} (Overspill) we get
  \begin{align*}
    &\neg \neg \, \exists \, e > \Nat. ~
    \M \vDash \neg \neg \, \sexists c \, \sforall u \sless e.
    ~ \alpha(u) \siff
    \Pi (u) \mid c
    \\ \implies ~
    &
    \neg \neg \, \exists c \col \M \, \forall u \col \Nat. ~
    \M \vDash
    \alpha(\num{u}) \siff
    \Pi (\num{u}) \mid c,
  \end{align*}
  where we used that since the equivalence holds for all \(u < e\) with \(e\) non-standard, it will in particular hold for all \(u \col \Nat \).
\end{proof}

\begin{lemma}[ ][Coding_nonstd_binary]\label{coding_model_binary}
  {If \( \standardElement{} \) is stable and \( \M \not \cong \Nat{} \), then for binary \( \alpha(x, y) \) and \( e \col \M \) we have}
  \[
    \neg \neg \, \exists c \col \M ~ \forall u \col \Nat.
    ~~ \M \vDash
    \alpha(\num{u}, e) \siff
    \Pi (\num{u}) \mid c.
  \]
\end{lemma}
\begin{proof}
  Analogous to the proof of \autoref{coding_model_unary}.
\end{proof}
These coding results allow us to connect a unary formula \( \alpha \) to an element \( c \col \M \) of the model, in such a way that the decidability of the divisibility for \( c \) will entail the decidability of \( \M \vDash \alpha (\num{\, \cdot \,}) \).

\section{Church's Thesis for First-Order Arithmetic}\label{section_CT}

Church's thesis (\CT) is an axiom of constructive mathematics which states that every total function is computable. We will assume a version of it in this paper, since by its addition to the ambient type theory, we merely need to show that a function can be defined at all, to prove its computability. This makes it possible to stay completely formal, yet achieve a textbook-style conciseness for proofs involving computability, even in their mechanization.

This addition is not possible in every meta-theory. 
If we were to add it to \ZF{}, it would immediately imply the computability of the function that solves the halting problem, leading to an inconsistent theory.
In general however, theories that tend to the constructive side do allow for the consistent addition of this axiom. In the type theory we use in this paper, this is achieved by strictly distinguishing between functional relations and total functions. 
The aforementioned function that solves the halting problem in \ZF{} can only be shown to be a functional relation, which means we can still safely assume total functions to be computable.
Currently, there is no consistency proof for \CT{} (\cf~Section 7.1 in~\cite{forster2021thesis}) and the exact type theory we are using, but there are proofs showing that it can be consistently added to very similar systems~\cite{yamada2020game, swan2019churchs}.

Since \CT{} makes reference to \emph{computability}, its exact form as an axiom does not only depend on the theory in which it is assumed, but also on the model of computation it makes reference to. 
Robinson's \Q{}, a finitely axiomatized arithmetical system, is expressive enough to serve as a computational model, and is a particularly well-suited choice in our case, leading us to the following formulation of \CT{} which we assume for the remainder of the paper:

\setCoqFilename{FOL.Tennenbaum.Church}
\begin{coqaxiom}[\CTQ][CT_Q]
    {For every function \( f : \Nat \to \Nat \) there exists a binary \( \Sigma_1 \)--formula \( \varphi_f(x,y) \)} such that for every \( n \col \Nat \) we have \( \Q \vdash \forall y. \, \varphi_f(\num n, y) \iff \num{f n} = y \).
\end{coqaxiom}
Note that \CTQ{} can be derived from the more conventional version of Church's thesis for \(\mu \)-recursive functions~\cite{kirst2023goedel}.\footnote{
  In~\cite{kirst2023goedel}, the abbreviation \CTQ{} was used to refer to a version of Church's thesis which applies to partial functions. From the partial version, the version for total functions can be derived, and in this paper we use \CTQ{} to refer to the latter, total version.
}
Using \CTQ{} we can get an internal representation \( \varphi_f \) of any computable function \( f \), allowing us to reason about the function inside of first-order arithmetic.
We also have an immediate use-case for \CTQ{}, since applying it to the injective prime function \( \primeFunc{} \) lets us settle the earlier \autoref{remark_prime_formula}:
\begin{coqfact}[ ][prime_div_relation]\label{represent_prime_function}
  There is a binary \( \Sigma_1 \)--formula representing the injective prime function \( \primeFunc{} \) in~\Q. \qed{}
\end{coqfact}
In Section~\ref{section_CTT} we defined decidable and enumerable predicates in a synthetic way, but using \CTQ{} we can now give characterizations and representations of such predicates by formulas in \Q (\cf~\cite{sep-goedel-incompleteness}).
\begin{definition}[ ][weak_repr]
  We call \( p : \Nat \to \Prop \) \coqlink[weak_repr]{\emph{weakly representable}} if there is a \( \Sigma_1 \)--formula \( \varphi_p(x) \) such that
  \( \forall n \col \Nat . \, p \, n \iff \Q \vdash \varphi_p(\num{n}) \),
  and \coqlink[strong_repr]{\emph{strongly representable}} if instead,
  \(  p \, n \to \Q \vdash \varphi_p(\num{n})  \)
  and
  \(  \neg p \, n \to \Q \vdash \neg \varphi_p(\num{n})  \)
  hold for every \( n \col \Nat \).
\end{definition}

\setCoqFilename{FOL.Tennenbaum.Church}
\begin{lemma}[Representability Theorem][CT_RTs]\label{representability_theorem}
  Assume \CTQ, and let \( p : \Nat \to \Prop \) be given.
  \begin{enumerate}
    \coqitem[CT_RTs]{If \( p \) is decidable, it is strongly representable.}
    \coqitem[CT_RTw]{If \( p \) is enumerable, it is weakly representable.}
  \end{enumerate}
\end{lemma}
\begin{proof}
  \setCoqFilename{FOL.Tennenbaum.SyntheticInType}
  If \( p \) is decidable, then there is a decider \( \Nat \to \Bool \) which can be used to define a function 
  \coqlink[Dec_decider_nat]{\( f \col \Nat \to \Nat \) such that \( \forall x \col \Nat. \, p \, x \leftrightarrow f x = 0 \)},
  and by \CTQ{} there is a binary \( \Sigma_1 \)--formula \( \varphi_f(x,y) \) representing \( f \).
  We then define \( \varphi_p (x) \cdef \varphi_f(x, \num{0}) \) and deduce
  \begin{alignat*}{19}
    p \, n
    &\implies&
    f n = 0
    &\implies&
    &\Q \vdash \num{f n} \sequal \num{0}
    &\implies
    &\Q \vdash \varphi_f(\num{n}, \num{0})
    &\implies
    &\Q \vdash \varphi_p (\num{n})
    \\
    \neg p \, n
    &\implies&
    f n \neq 0
    &\implies&
    &\Q \vdash \sneg (\num{f n} \sequal \num{0})
    &\implies
    &\Q \vdash \sneg \varphi_f(\num{n}, \num{0})
    &\implies
    &\Q \vdash \sneg \varphi_p (\num{n})
  \end{alignat*}
  which shows that \( p \) is strongly representable.

  If \( p \) is enumerable, then there is an enumerator \( \Nat \to \Opt \Nat \) which can be used to define a function 
  \coqlink[enumerable_nat]{\( f : \Nat \to \Nat \) such that \( \forall x \col \Nat. \, p \, x \leftrightarrow \exists n. \, f n = x+1 \)}.
  By \CTQ{} there is a binary \( \Sigma_1 \)--formula \( \varphi_f(x,y) \) representing \( f \).
  We then define \( \varphi_p (x) \cdef \sexists n. \, \varphi_f(n, \natS x)  \) giving us
  \begin{align*}
    \Q \vdash \varphi_p (\num{x})
    ~&\Longleftrightarrow~
    \Q \vdash \sexists n. \, \varphi_f(n, \natS \num{x})
    \Longleftrightarrow~
    \exists n \col \Nat. \, \Q \vdash \varphi_f(\num{n}, \natS \num{x})
    \\&\Longleftrightarrow~
    \exists n \col \Nat. \, \Q \vdash \num{f n} \sequal \natS \num{x}
    \Longleftrightarrow~
    \exists n \col \Nat. \, f n = x + 1
    ~\Longleftrightarrow~
    p \, x
  \end{align*}
  which shows that \( p \) is weakly representable by a \( \Sigma_1 \)--formula.
\end{proof}

\section{Tennenbaum's Theorem}\label{section_Tennenbaum}

With our choice of \( \CTT + \CTQ \) for the meta-theory in place, we now begin with our analysis of Tennenbaum's theorem. 
We will present several proofs of the theorem from the literature.
In a classical meta-theory all of these proofs would yield the same result, but in our constructive setting, they turn out to differ in the strength of their assumptions and conclusions.
Almost all the proofs will make use of some coding results for non-standard models from Section~\ref{section_coding}, enabling us to use a single model element to fully encode the standard part of any predicate \( p : \M \to \Prop \).

For the proof in Section~\ref{section_Tennenbaum_diagonal} we will assume enumerability of the model, enabling a very direct diagonal argument~\cite{boolos2002computability}.
In Section~\ref{section_Tennenbaum_inseparable} we look at the proof approach that is most prominently found in the literature~\cite{smith2014tennenbaum, kaye2011tennenbaum} and uses the existence of recursively inseparable sets.

Another refinement of this proof was proposed in a post by Makholm~\cite{SEmakholm} and comes with the advantage that it circumvents the usage of Overspill.
Strikingly, it turns out that in the constructive setting, this eliminates the necessity for \MP, which is required for the standard proof using inseparable sets.
Additionally, we look at the consequences of Tennenbaum's theorem once the underlying semantics is made explicitly constructive.
The latter two variations are discussed in Section~\ref{section_Variants}.

\subsection{Via a Diagonal Argument}\label{section_Tennenbaum_diagonal}

We start by noting that every \HA{} model can prove the most basic fact about division and remainders.
\setCoqFilename{FOL.Tennenbaum.Peano}
\begin{lemma}[Euclid's Lemma][iEuclid]\label{lemma_Euclid}
  {We have}
  \[
    \M \vDash \sforall e \, d ~ \sexists r \, q. ~ e \sequal q \cdot d + r
    ~\land~
    (0 < d \to r < d)
  \]
  and the {uniqueness property} telling us that if \( r_1, r_2 < d \) then \( q_1 \cdot d + r_1 = q_2 \cdot d + r_2 \) implies \( q_1 = q_2 \) and \( r_1 = r_2 \).
\end{lemma}
\begin{proof}
  For Euclid's lemma, there is a standard proof by induction on \( e \).
  The uniqueness claim requires some basic results about the strict order.
\end{proof}

\needspace{4\baselineskip}
\setCoqFilename{FOL.Tennenbaum.Tennenbaum_diagonal}
\begin{lemma}[ ][dec_div]\label{decidable_div}
  {If \( \M \) is enumerable and discrete, then \( \fun (n \col \Nat)  (d \col \M). \, \M \vDash \num{n} \! \mid \! d \) has a decider.}
\end{lemma}
\begin{proof}
  Let \( n \col \Nat \) and \( d \col \M \) be given.
  By \autoref{lemma_Euclid} we have \( \exists q', r' \col \M. \, d = q' \cdot \num n + r' \).
  This existence is propositional, so presently we cannot use it to give a decision for \( \num n \mid d \).
  \setCoqFilename{FOL.Tennenbaum.SyntheticInType}
  \coqlink[enumerable_surjection]{Since \M is enumerable, there is a surjective function \( g : \Nat \to \M \)} 
  and the above existence therefore shows \( \exists q, r \col \Nat. \, d = (g \, q) \cdot \num n + (g \, r) \).
  Since equality is decidable in \M and 
  \setCoqFilename{FOL.Tennenbaum.MoreDecidabilityFacts}
  \coqlink[Witnessing_natnat]{\( \Nat^2 \) is witnessing}, 
  we get \( \Sigma q, r \col \Nat. \, d = (g \, q) \cdot \num n + (g \, r) \), giving us computational access to \( r \), now allowing us to construct the decision.
  By the uniqueness part of \autoref{lemma_Euclid} we have \( g \, r = 0 \iff \num n \mid d \), so the decidability of \( \num n \mid d \) is entailed by the decidability of \( g \, r = 0 \).
\end{proof}

\setCoqFilename{FOL.Tennenbaum.Tennenbaum_diagonal}
\begin{lemma}[ ][Stable_std_stable_stdModel]\label{MP_std_stable}
  \begin{enumerate}
    \coqitem[Stable_std_stable_stdModel]{If \( \standardElement{} \) is stable, then \( \M \cong \Nat \) is stable.}
    \coqitem[Stable_std_negn_nonStd__sdtMobel_equiv]{If \( \standardElement{} \) is stable, then \( \neg \M > \Nat \) is equivalent to \( \M \cong \Nat \).}
    \coqitem[MP_Discrete_stable_std]{Assuming \MP{}, if \M is discrete, then \( \standardElement{} \) is stable}.
  \end{enumerate}
\end{lemma}
\begin{proof}
  \( (1) \) is trivial by \autoref{std_model_equiv_definitions}.
  \( (2) \) The implication \( \forall e. \, \standardElement{\, e} \to \neg \exists e. \, \neg \standardElement{\, e} \) holds constructively, but the converse needs the stability of \( \standardElement{} \).
   For \( (3) \), recall that \( \standardElement{\, e} \) stands for \( \exists n \col \Nat. \, \num{n} = e \).
  Since \( \num{n} = e \) in \M is decidable, stability follows from \autoref{Nat_witnessing}.
\end{proof}
\begin{lemma}[ ][Coding_Dec]\label{decidable_coding}
  {If \( \standardElement{} \) is stable, \( \M \not \cong \Nat \), and \( p \col \Nat \to \Prop \) decidable, then potentially there} is a code \( c \col \M \) such that
  \(
    \forall n \col \Nat. ~
    p \, n \iff \M \vDash \num{\primeFunc{n}} \mid c
  \).
\end{lemma}
\begin{proof}
  By \autoref{representability_theorem}, there is a formula \( \varphi_p \) strongly representing \( p \).
  Under the given assumptions, we can use the coding \autoref{coding_model_unary}, yielding a code \( c \col \M \) for the formula \( \varphi_p \), such that \( \forall u \col \Nat. \, \M \vDash \varphi_p(\num{u}) \iff \Pi(\num{u}) \mid c \).
  Overall this shows:
  \begin{alignat*}{9}
    p \, n
    &\implies
    \Q \vdash \varphi_p(\num{n})
    &\implies&
    &\M \vDash \varphi_p(\num{n})
    &\implies&
    &\M \vDash \Pi(\num{n}) \mid c
    \\
    \neg p \, n
    &\implies
    \Q \vdash \neg \, \varphi_p(\num{n})
    &\implies&
    \neg \, &\M \vDash \varphi_p(\num{n})
    &\implies&
    \neg \, &\M \vDash \Pi(\num{n}) \mid c.
  \end{alignat*}
  Since \( p \) is decidable, the latter implication entails \( \M \vDash \Pi(\num{n}) \mid c \implies p \, n \), which overall shows the desired equivalence.
\end{proof}
This gives us the following version of Tennenbaum's theorem:
\begin{theorem}[ ][Tennenbaum_diagonal_ct]\label{Tennenbaum_diagonal}
  {Assuming \MP{} and discrete \( \M \), enumerability of \( \M \) implies \( \M \cong \Nat \).}
\end{theorem}
\begin{proof}
  By \autoref{MP_std_stable} it suffices to show \( \neg \neg \M \cong \Nat \).
  So assume \( \M \not \cong \Nat \) and try to derive \( \bot \).
  \setCoqFilename{FOL.Tennenbaum.SyntheticInType}
  \coqlink[enumerable_surjection]{Given the enumerability, there is a surjective function \( g \col \Nat \to \M \)}, 
  allowing us to define the predicate \( p \cdef \fun n \col \Nat. \, \neg \, \M \vDash \num{\primeFunc{n}} \mid g \, n \), which is decidable by \autoref{decidable_div}.
  By the coding result in \autoref{decidable_coding} there is an \(e \col \M \) which codes \( p \), and by the surjectivity of \( g \), there is some \( c \col \Nat \) with \(g \, c = e\). Combined, these facts give us
  \begin{align*}
    \neg \, \M \vDash \num{\primeFunc{c}} \mid g \, c
    ~\stackrel{\text{def.}}{\Iff}~ p \, c
    ~\stackrel{\text{\text{coding}}}{\Iff}~  \M \vDash \num{\primeFunc{c}} \mid g \, c
  \end{align*} 
  leading to the desired contradiction.
\end{proof}

\subsection{Via Inseparable Predicates}\label{section_Tennenbaum_inseparable}

The most frequently reproduced proof of Tennenbaum's theorem~\cite{kaye2011tennenbaum, smith2014tennenbaum} uses the existence of recursively inseparable sets and non-standard coding to establish the existence of a non-recursive set.

\setCoqFilename{FOL.Tennenbaum.Tennenbaum_insep}
\begin{definition}[ ][Insep_pred]\label{def_inseparable_predicates}
  A pair \( A, B : \Nat \to \Prop \) of predicates is called {\emph{inseparable}} if they are disjoint and \( A \incl D \), \( B \incl \neg D \) implies the undecidability of \( D \).
\end{definition}

\needspace{4\baselineskip}
\begin{lemma}[ ][CT_Insep_pred]\label{lemma_inseparable_predicates}
  {There are inseparable enumerable predicates \( A, B : \Nat \to \Prop \).}
\end{lemma}
\begin{proof}
  We use an \coqlink[surj_form_]{enumeration \( \enumform{n} : \Form \) of formulas} to define predicates \( A \cdef \lambda n \col \Nat. \, \Q \vdash \neg \, \enumform{n}(\num{n})  \) and \( B \cdef \lambda n \col \Nat. \, \Q \vdash \enumform{n}(\num{n}) \), which are disjoint since \( \Q \not \vdash \bot \), and enumerable since \coqlink[enumerable_Q_prv]{proofs over \Q can be enumerated}.
  Given a predicate \( D \) with \( A \incl D \), \( B \incl \neg D \) and assuming it were decidable, \autoref{representability_theorem} gives us a formula strongly representing \( D \), and by the enumeration there is \( d \col \Nat \) such that \( \enumform d \) is said formula.
  Everything together gives us the following chain of implications:
  \begin{align*}
    D \, d 
    \stackrel{\text{s.repr.}}{\implies} \Q \vdash \Phi_d(\num d)
    \stackrel{\text{def.}}{\implies} B \, d
    \stackrel{\incl}{\implies} \neg D \, d
    \stackrel{\text{s.repr.}}{\implies} \Q \vdash \neg \Phi_d(\num d)
    \stackrel{\text{def.}}{\implies} A \, d
    \stackrel{\incl}{\implies} D \, d
  \end{align*}
  Since this shows \( D \, d \Iff \neg D \, d \), we must conclude that \( D \) is undecidable.
\end{proof}
\begin{corollary}[ ][CT_Insep_form]\label{inseparable_formulas}
  There is a pair \( \alpha, \beta \) of unary \( \Sigma_1 \)--formulas such that \( A \cdef \lambda n \col \Nat. \, \Q \vdash \alpha(\num{n}) \) and \( B \cdef \lambda n \col \Nat. \, \Q \vdash \beta(\num{n}) \) are inseparable and enumerable.
\end{corollary}
\begin{proof}
  Use weak representability (\autoref{representability_theorem}) on the predicates given by \autoref{lemma_inseparable_predicates}.
\end{proof}
Contrary to the proof delineated in Section~\ref{section_Tennenbaum_diagonal}, the alternative proof via inseparable sets eliminates the need for enumerability of the model. 
Furthermore, we will now factor Tennenbaum's theorem into two parts, mirroring what Kaye does in~\cite{kaye2011tennenbaum}. 
He first shows that for any non-standard model \M, a non-recursive coded set exists within \M, and separately to this, he establishes that computability of the model would entail the recursiveness of all coded sets.
Combining both of these results then yields Tennenbaum's theorem.
In the following we will focus on the first part, and establish the existence of an undecidable predicate, which is coded by an element of \M.
Our coding here concretely uses prime numbers, but in Section~\ref{section_root_of_Tennenbaum}, we will revisit a more abstract perspective on the coding.
\begin{definition}[ ][Div_nat]
  For \( d \col \M \) define the predicate \( \divNat d \cdef \fun n \col \Nat. \, \M \vDash \num{n} \mid d \).
\end{definition}

\begin{theorem}[ ][Tennenbaum_inseparable]\label{tennenbaum-inseparable}
  Assuming stability of \(\standardElement{}\), then \( \M > \Nat \) implies \( \neg \neg \exists d \col \M. \, \neg \Dec{(\divNat d)} \).
\end{theorem}
\begin{proof}
  We will assume \( \neg \exists d. \, \neg \Dec{(\divNat d)} \) and try to reach a contradiction.
  By \autoref{inseparable_formulas} there is a pair \( \alpha', \beta' \) of inseparable unary \( \Sigma_1 \)--formulas. 
  By
  \setCoqFilename{FOL.Incompleteness.sigma1}
  \coqlink[8eb9c83d193eda350750450f09b3dfb3]{\( \Sigma_1 \)--compression}, 
  they can be written in the form \( \exists w. \, \alpha(w, x) \) and \( \exists w. \, \beta(w, x) \), where \( \alpha, \beta \) are \( \Delta_1 \).
  Since they are disjoint, we have:
  \[
    \Nat \vDash \sforall x \, w \, v \sless \num{n}. ~ \alpha(w, x) \simp \beta(v,x) \simp \bot
  \]
  for every bound \( n \col \Nat \). 
  Due to its bounded quantification, the above formula is also \( \Delta_1 \), allowing us to use \( \Delta_1 \)--absoluteness (\autoref{Delta1_absoluteness}) to get
  \[
    \M \vDash \sforall x \, w \, v \sless \num{n}. ~ \alpha(w, x) \simp \beta(v,x) \simp \bot
  \]
  Since we are assuming \( \M > \Nat \), we have \( \M \not \cong \Nat \), which together with the stability of \(\standardElement{} \) allows us to use both Overspill (\autoref{Overspill}) and coding for predicates (\autoref{coding_model_binary}).
  Overspill gives us the potential existence of an element \( e \col \M > \Nat \) with
  \[
    \M \vDash \sforall x \, w \, v \sless e. ~ \alpha(w, x) \simp \beta(v,x) \simp \bot
  \]
  which shows the disjointness of \( \alpha, \beta \) when everything is bounded by \( e \).
  By the coding result, we can get a code \( c \col \M \) which satisfies
  \begin{align*}
    \forall u \col \Nat. ~ \M \vDash (\sexists w \sless e. \, \alpha(w, \num{u})) \siff \Pi(\num u) \mid c
  \end{align*}
  Given the above equivalence between \( D \cdef \lambda n \col \Nat. \, \M \vDash \sexists w \sless e. \, \alpha(w, \num{n}) \) and \( \Pi(\num{\, \cdot \,}) \mid c \), our initial assumption \( \neg \exists d. \, \neg \Dec{(\divNat d)} \) entails that \( D \) cannot be undecidable.
  However, we will now see that \( D \) separates the given inseparable formulas: 
  \begin{align*}
    \text{(1)} ~~~~~~~ \Q \vdash \sexists w. \, \alpha(w, \num{\, \cdot \,})
    ~& \incl~ D \\
    \text{(2)} ~~~~~~~ \Q \vdash \sexists w. \, \beta(w, \num{\, \cdot \,})
    ~& \incl~ \neg D
  \end{align*}
  which will establish its undecidability.
  \begin{enumerate}
    \item If \( \Q \vdash \sexists w. \, \alpha(w, \num{n}) \) there is \( w \col \Nat \) with \( \Nat \vDash \alpha(\num{w}, \num{n}) \) and \( \M \vDash \alpha(\num{w}, \num{n}) \) by \autoref{Delta1_absoluteness}.
    Since \( \num w < e \) we can therefore show \( D \, n \).
    \item If \( \Q \vdash \sexists w. \, \beta(w, \num{n}) \) there is \( w \col \Nat \) with \( \Nat \vDash \beta(\num{w}, \num{n}) \) and \( \M \vDash \beta(\num{w}, \num{n}) \) by \autoref{Delta1_absoluteness}.
    Assume we had \( D \, n \), then there is \( v < e \) with \( \M \vDash \alpha(v, \num n) \), which leads to a contradiction since \( \alpha, \beta \) were shown disjoint below \( e \).
  \end{enumerate}
  It leaves us with the contradiction that \( D \) is both potentially decidable and undecidable.
\end{proof}

\subsection{Variants of the Theorem}\label{section_Variants}

\setCoqFilename{FOL.Tennenbaum.Variants}
We now investigate two further variants of the theorem, going back to McCarty~\cite{mccarty1987variations, McCarty} and Makholm~\cite{SEmakholm} respectively.
They make use of a stronger notion of inseparable formulas, which requires the formulas to be provably disjoint on the object level.
\setCoqFilename{FOL.Tennenbaum.Variants}
\begin{definition}[\HA-inseparable][def_HA_Insep]\label{def_HA_inseparable}
  A pair of unary formulas \( \alpha(x), \beta(x) \) is called {\emph{\HA-inseparable}} if they are disjoint in the sense of \( \HA \vdash \sneg \sexists x. \, \alpha(x) \sand \beta(x) \) and if any \(D\) with \( \Q \vdash \alpha(\num{\, \cdot \,}) \incl D \), \( \Q \vdash \beta(\num{\, \cdot \,}) \incl \neg D \) is undecidable.
\end{definition}
According to McCarty~\cite{McCarty}, the existence of \HA-inseparable formulas can be established by taking the construction of inseparable formulas as seen in \autoref{lemma_inseparable_predicates}, and internalizing the given proof within \HA{}\@.
However, as pointed out in~\cite{peters2022goedel} (Fact 6.1), \emph{Rosser's trick}~\cite{rosser1937trick} can be used to construct the desired \HA-inseparable formulas from the inseparable formulas given by \autoref{inseparable_formulas}. 
\setCoqFilename{FOL.Tennenbaum.HA_insep}
\begin{definition}[Rosser Formula][rosser]
  Given any binary formulas \( \alpha(t, x) \) and \( \beta(t ,x) \) we define
  \begin{align*}
    (\alpha < \beta) (x) \cdef \sexists t. \, \alpha(t, x) \sand \sforall v. \, \beta(v, x) \to t < v
  \end{align*}
\end{definition}
Intuitively, if we interpret \( \exists t. \, \alpha(t, x) \) as ``\textit{There is some time \(t\) at which \( \alpha(t, x) \) can be verified}'', then \( (\alpha < \beta)(x) \) expresses that \( \alpha(\, \cdot \,, x) \) will be verified before \( \beta(\, \cdot \,, x) \).

\begin{lemma}[Disjointness of Rosser Formulas][rosser_disj]\label{appendix_rosser_disjoint}
  \( \HA \vdash \forall x. \, (\alpha < \beta)(x) \to (\beta < \alpha)(x) \to \bot \)
\end{lemma}
\begin{proof}
  Our goal is to prove \( (\alpha < \beta)(x), \, (\beta < \alpha)(x), \, \HA \vdash \bot \). From the formulas in the context of the derivation, we can conclude that there are terms \(t, t'\) such that
  \begin{align*}
    \alpha(t, x) ~\sand~ &\sforall v. \, \beta(v, x) \to t < v \\
    \beta(t', x) ~\sand~ &\sforall v. \, \alpha(v, x) \to t' < v
  \end{align*}
  From this we get \( t' < t \sand t < t' \) and then clearly \( \HA \vdash t < t' \to t' < t \to \bot \).
\end{proof}

\needspace{4\baselineskip}
\setCoqFilename{FOL.Tennenbaum.Variants}
\begin{lemma}[\HA-inseparable formulas][HA_Insep]\label{lemma_HA_inseparable}
  There are unary \HA-inseparable formulas.
\end{lemma}
\begin{proof}
  By \autoref{inseparable_formulas}, there are inseparable \( \Sigma_1 \)--formulas, and due to 
  \setCoqFilename{FOL.Incompleteness.sigma1}
  \coqlink[8eb9c83d193eda350750450f09b3dfb3]{\( \Sigma_1 \)--compression}, 
  we can assume they have the form \( \alpha'(x) = \sexists w. \, \alpha(w, x)  \), \( \beta'(x) = \sexists w. \beta(w, x) \), where \( \alpha, \beta \) are both \( \Delta_1 \).  
  Since we have the equivalence
  \begin{align*}
    \forall v. \, \beta(v, x) \to w < v
    ~\iff~
    \forall v \leq w. \neg \beta(v, x)
  \end{align*}
  where the latter formula is \( \Delta_1 \), we can conclude that the Rosser formula \( (\alpha < \beta)(x) \) is \( \Sigma_1 \).

  We will now show that 
  \setCoqFilename{FOL.Tennenbaum.HA_insep}
  \coqlink[rosser_inherit]{\( \Q \vdash \alpha'(\num n) \) implies \( \Q \vdash (\alpha < \beta)(\num n) \)}. From the assumption and soundness we get \( w \col \Nat \) such that \( \Nat \vDash \alpha(\num w, \num n) \).
  Since \( \alpha', \beta' \) are inseparable and therefore disjoint, we must have \( \Nat \vDash \neg \beta(\num v, \num n) \) for every \( v \col \Nat \) and therefore in particular \( \Nat \vDash \forall v. \, \beta(v, \num n) \to \num w < v \). 
  Overall this shows \(\Nat \vDash (\alpha < \beta)(\num n) \), which by \( \Sigma_1 \)--completeness gives \( \Q \vdash (\alpha < \beta)(\num n) \).
  With the exact same reasoning we can get the analogous result for \( \beta < \alpha \).

  We can now move on to prove that \( \alpha < \beta \) and \( \beta < \alpha \) constitute a pair of \HA-inseparable formulas. The disjointness property follows from the previous \autoref{appendix_rosser_disjoint}. Abbreviating any predicate \( \lambda n \col \Nat. \, \Q \vdash \varphi(\num n) \) by \( \Q \vdash \varphi \),
  it remains to show that any \( D \) satisfying 
  \begin{align*}
    \Q \vdash (\alpha < \beta) 
    ~&\incl~ D \\
    \Q \vdash (\beta < \alpha) 
    ~&\incl~
    \neg D
  \end{align*}
  is undecidable. By the implications we have just shown above, we get the inclusions
  \begin{align*}
    \Q \vdash \alpha'  
    ~\incl~
    \Q \vdash (\alpha < \beta) 
    ~&\incl~
    D \\
    \Q \vdash \beta'  
    ~\incl~
    \Q \vdash (\beta < \alpha) 
    ~&\incl~
    \neg D
  \end{align*}
  The undecidability of \( D \) therefore follows from the inseparability of the pair \( \alpha', \beta' \).
\end{proof}
This stronger disjointness property allows us to easily establish that satisfiability of \HA{}-inseparable formulas is undecidable in any model.
\begin{lemma}[ ][HA_Insep_undec]\label{lemma_HA_insep_undecidable}
  {If \( \alpha, \beta \) are \HA-inseparable, then \( \satformula{\M}{\alpha} \) and \( \satformula{\M}{\beta} \) are undecidable.}
\end{lemma}
\begin{proof}
  Using soundness and the \HA{}-disjointness of \( \alpha \) and \( \beta \),  we get
  \begin{align*}
    \Q \vdash \alpha(\num{n})
    ~&\stackrel{\text{sound}}{\Longrightarrow}~
    \M \vDash \alpha(\num{n}) \\
    \Q \vdash \beta( \num{n} )
    ~&\stackrel{\text{sound}}{\Longrightarrow}~
    \M \vDash \beta( \num{n} )
    ~\stackrel{\HA \text{-disj.}}{\Longrightarrow}~
    \neg \M \vDash \alpha( \num{n} )
  \end{align*}
  Undecidability of \( \satformula{\M}{\alpha}\) then follows from inseparability of the given formulas (\autoref{inseparable_formulas}), and the same argument also shows the undecidability of \( \satformula{\M}{\beta}\).
\end{proof}

\setCoqFilename{FOL.Tennenbaum.Variants}
McCarty~\cite{mccarty1987variations, McCarty} considers Tennenbaum's theorem with constructive semantics.
Instead of models placed in classical set theory, he works in an intuitionistic theory (e.g. \IZF), making the interpretation of the object-level disjunction much stronger.
By furthermore assuming \MP{}, he is then able to show that all models of \HA{} in this constructive setting are standard.
To achieve the constructive rendering of disjunctions, we will locally make use of the following choice principle:
\setCoqFilename{FOL.Tennenbaum.SyntheticInType}
\begin{definition}[ ][UC]\label{Unique_Choice}
  By \( \UC \) we denote the principle of unique choice:
  \[
    \forall X \, Y \, (R \col X \to Y \to \Prop). \, (\forall x \, \exists ! y. \, R x y) \to \exists f \col X \to Y. \, \forall x. \, R x (f x)
  \]
\end{definition}
Note that generally, Church's thesis and unique choice principles combined prove the negation of \LEM{} (\cf~\cite{forster:LIPIcs:2021:13455}), which will make the results that use \UC{} (deliberately) anti-classical.
\setCoqFilename{FOL.Tennenbaum.Variants}
\begin{lemma}[ ][UC_unary_DN_Dec]\label{lemma_McCarty}
  If \UC{} and \( \M > \Nat \), then there is no unary formula \( \alpha \) such that \( \satformula{\M}{\alpha} \) is undecidable.
\end{lemma}
\begin{proof}
  Since single instances of the law of excluded middle are provable under double negation, induction on \( n \) can be used to prove \( \M \vDash \sforall n. \, \sneg \sneg \sforall x < n. \, \alpha(x) \sor \sneg \alpha(x) \).
  Choosing \(e \col \M > \Nat \) for the bound \(n\) above, we get \( \M \vDash \neg \neg \sforall x < e. \, \alpha(x) \sor \sneg \alpha(x) \) and therefore in particular \( \neg \neg \, \forall n \col \Nat. \, \M \vDash \alpha(\num n) \sor \sneg \alpha(\num n) \), meaning \( \satformula{\M}{\alpha} \) is potentially definite.
  \setCoqFilename{FOL.Tennenbaum.SyntheticInType}
  Since \coqlink[UC_Def_Dec]{\UC{} entails that definite predicates on \Nat{} are decidable}, 
  \setCoqFilename{FOL.Tennenbaum.Variants}
  we get \( \neg \neg \Dec{(\satformula{\M}{\alpha})} \) for all formulas \( \alpha \), which implies the non-existence claim.
\end{proof}
\begin{corollary}[ ][UC_no_nonStd]\label{corollary_auc_nonStd}
  If \UC{}, then \( \neg \M > \Nat \).
\end{corollary}
\begin{proof}
  By \autoref{lemma_HA_inseparable} there exist \HA{}-inseparable formulas, and by \autoref{lemma_HA_insep_undecidable} their satisfiability is undecidable.
  This directly contradicts \autoref{lemma_McCarty}.
\end{proof}
\begin{theorem}[McCarty][McCarty]\label{McCarty}
  Given \UC{} and \MP{}, every model of \HA{} is standard.\footnote{Or another way to put it: \HA{} is \emph{categorical}.}
\end{theorem}
\begin{proof}
  By \autoref{corollary_auc_nonStd} we have \( \neg \M > \Nat \).
  Since \( \HA \vdash \forall x y. \, x \sequal y \sor \sneg \, x \sequal y \), it follows from \( \UC \) that every model \( \M \vDash \HA \) is discrete.
  Since we also assume \MP{}, \autoref{MP_std_stable} \( (2) \) now entails \( \M \cong \Nat \).
\end{proof}

Now turning to Makholm~\cite{SEmakholm}, we will no longer require \UC{}, but will instead make use of the fact that the coding lemma can be established in \HA{}.
While we have proven this result on the level of Coq in \autoref{coding_nat}, we did not mechanize its derivation in the first-order theory \HA{}.
To make this transparent, we make its assumption explicit here.
\begin{coqhyp}[\HA-Coding][obj_Coding]\label{hypothesis_object_level_coding}
  {For any unary formula \( \alpha(x) \), \HA{} can internally prove} the coding lemma:
  \(
    \HA \vdash \sforall n \, \neg \neg \sexists c \, \sforall u \sless n. ~
    \alpha(u) \siff \Pi (u) \mid c
  \).\footnote{
      In the conference paper~\cite{hermes2022tennenbaum}, Hypothesis 2 stated that \( \HA \vdash \sforall n \, \sexists c \, \sforall u \sless n. ~
      \alpha(u) \siff \Pi (u) \mid c \). 
      Contrary to what was claimed, this is not provable for arbitrary formulas \( \alpha \).
      The new formulation of \HA{}-coding solves this mistake through the addition of the double negation.
    }
\end{coqhyp}

\begin{theorem}[Makholm][Makholm]\label{Makholm}
  If \( \M > \Nat \) then \( \neg \neg \exists d \col \M. \, \neg \Dec{(\divNat d)} \).
\end{theorem}
\begin{proof}
  By \autoref{lemma_HA_inseparable} there are \HA-inseparable unary formulas \( \alpha \) and \( \beta \).
  Using soundness and \HA-coding we get that \( \alpha \) can be coded up to any bound \(n \col \M \)
  \[
    \M \vDash \forall n \, \neg \neg \sexists c \, \sforall u \sless n. \,
    \alpha(u) \siff \Pi (u) \mid c.
  \]
  By assumption, there is a non-standard element \( e \col \M > \Nat \), and picking this for the bound \(n\), we potentially get a code \( c \col \M \) satisfying
  \[
    \M \vDash \sforall u \sless e. \,
    \alpha(u) \siff \Pi (u) \mid c.
  \]
  Since the above equivalence holds for all standard numbers \( u \col \M \), the undecidability of \( \M \vDash \alpha (\num{\, \cdot \,}) \) (\autoref{lemma_HA_insep_undecidable}) entails \( \neg \Dec{ (\divNat{c}) }\), which shows the existence claim.
\end{proof}
Note the quite remarkable fact that in contrast to \autoref{tennenbaum-inseparable}, we do not need to assume \MP{} or discreteness of the model in order to establish \autoref{Makholm}.

\subsection{Unearthing the Roots of Tennenbaum's Theorem}\label{section_root_of_Tennenbaum}

In this section, we want to record briefly and mostly informally some general ingredients that seem to go into a proof of Tennenbaum's theorem.

Taking a look at the proofs of both McCarty and Makholm, we can see a very clear structure.
In both of them:
\begin{itemize}
  \item \HA-inseparable formulas are used to derive the existence of an undecidable \(\satformula{\M}{\alpha}\).
  \item Assuming the model has a non-standard element, it is shown that there potentially exists an element with undecidable divisibility on numerals.
\end{itemize}
In Makholm's proof the latter point was achieved by usage of the coding result \autoref{hypothesis_object_level_coding}, which establishes a connection between satisfaction of a formula and divisibility with respect to its code number.
Close inspection of the proof of \autoref{coding_nat} reveals that only two properties of divisibility by primes are needed for it:
\begin{align*}
  \forall x. \, \neg \, \pi_x \mid 1,
  \hspace{1.2cm}
  \forall n \, c \, \exists c' \, \forall x. ~ \pi_x \mid c' \iff x = n \lor \pi_x \mid c.
\end{align*}
We can abstract away from divisibility and formulate the result as:
\setCoqFilename{FOL.Tennenbaum.Abstract_coding}
\begin{lemma}[ ][DN_coding]
  Let \( \in~ \col \Nat \to \Nat \to \Prop \) be a binary predicate satisfying the conditions
  \begin{enumerate}
    \item \( \exists e \, \forall x. \, \neg \, x \in e \) 
    \hspace{10em} (Empty set)
    \item \( \forall n \, c \, \exists c' \, \forall x. \, x \in c' \iff x = n \lor x \in c \) 
    \hspace{1.5em} (Extension),
  \end{enumerate}
  which we can interpret as axiomatizing a weak notion of sets. 
  For any predicate \( p \col \Nat \to \Prop \) we then have \( \forall n \, \neg \neg \, \exists c \, \forall u < n. \, p \, u \iff u \in c \).
\end{lemma}
\begin{proof}
  Following what we did for \autoref{coding_nat}, it is easier to prove the strengthened statement
  \[
    \neg \neg \, \exists c \col \Nat{} ~ \forall u \col \Nat.
    \big(u < n \rightarrow (p \, u \leftrightarrow u \in c) \big) \land \big(u \in c \rightarrow u < n \big)
  \]
  by induction on \(n\).
  The proof then follows the proof of \autoref{coding_nat}. In the base case we make use of the empty set, and in the induction step we do a case distinction on \( \neg \neg (p \, n \lor \neg p \, n) \) and utilize the extension property.
  \coqlink[coding_bounded]{The details for these steps were checked with  Coq}.
\end{proof}
If we have a binary formula \( \varphi_\in(x,y) \) satisfying the same conditions inside of \HA{}, we can give a derivation of \( \HA \vdash \forall n \, \neg \neg \, \exists c \, \forall u < n. \, \alpha(u) \iff \varphi_\in(u, c) \) and verbatim run Makholm's proof.
Using \( \num{\, \cdot \,} \in d \) as a shorthand for \( \lambda n. \, \M \vDash \varphi_\in(\num n, d) \) we would then get:
\begin{quote}
  \hspace{-1.4cm}
  \textbf{Theorem.}
  \textit{If \( \M > \Nat \) then \( \neg \neg \exists d \col \M. \, \neg \Dec{(\num{\, \cdot \,} \in d)} \).}
\end{quote}
This highlights that the statement of \autoref{Makholm} is not inherently tied to divisibility, which is also remarked by Kaye in~\cite{kaye2011tennenbaum}.
It rather seems generally tied to relations that allow the implementation of finite sets or sequences inside of \HA{}.
Visser~\cite{visser2008pairs} analyzes several axiomatizations of pairs, sets and sequences in first-order theories, and the conditions we have listed above appear as the axioms of the weak set theory \WS{}, which can interpret \Q and is therefore essentially undecidable.
This latter point raises an interesting question:
One can now wonder if it is decidability of \( \num{\, \cdot \,} \in d \) or \( \divNat{d} \) respectively, together with the essential undecidability of \WS{} and \Q{}, which combine to rule out that the model is non-standard.

\section{Discussion}\label{section_discussion}

\subsection{General Remarks}

In Section~\ref{section_Tennenbaum}, we presented several proofs of Tennenbaum's theorem which we summarize in the table below, listing their assumptions\footnote{
  We do not list the global assumption \CTQ{}\@.
  \HA-coding (\autoref{hypothesis_object_level_coding}) was not mechanized in Coq but is provable, which is why we leave it out of the table.
}
on the left and the conclusion on the right.
\begin{figure}[h!]
  \begin{tabular*}{0.97\textwidth}{c|c|c||rc|l}
    \hline
    \hyperref[classical_principles]{\textcolor{cyan}{\MP}} & 
    \hyperref[Unique_Choice]{\textcolor{cyan}{\( \UC \)}} & 
    \hyperref[type_properties]{\textcolor{cyan}{\M discrete}} & 
    & \hspace{-8em} Theorem 
    & Idea/Technique 
    \\ \hline
    \( \bullet \) & 
    & 
    \( \bullet \) &  
    \M enumerable & 
    \hspace{-3.8em}\( \to \) \hspace{2em} \( \M \cong \Nat \) & 
    \hyperref[Tennenbaum_diagonal]{Diagonalization}
    \\ \hline
    \( \bullet \) & 
    & 
    \( \bullet \) & 
    \( \M > \Nat \) &
    \( \to \neg \neg \exists d \col \M. \neg \Dec{(\divNat d)} \) & 
    \hyperref[tennenbaum-inseparable]{Inseparability}
    \\ \hline
    & 
    & 
    & 
    \( \M > \Nat \) &
    \( \to \neg \neg \exists d \col \M. \neg \Dec{(\divNat d)} \) & 
    \hyperref[Makholm]{Makholm}
    \\ \hline
    \( \bullet \) & 
    \( \bullet \) & 
    & 
    &
    \( \M \cong \Nat \) & 
    \hyperref[McCarty]{McCarty}
    \\ \hline
  \end{tabular*}
\end{figure}

\noindent
First note that in the first row we can clearly show the reverse implication (\( \M \cong \Nat \) implies enumerability).
If we additionally assume \MP{} and discreteness of \M, the first two results combined give the equivalences
\begin{align*}
  \M \text{ enumerable}
  ~\Iff~
  \M \cong \Nat
  \stackrel{\ref{MP_std_stable}}{~\Iff~}
  \neg \, \M > \Nat
  ~\Iff~
  \neg \exists d. \neg \Dec{(\divNat d)}.
\end{align*}

Furthermore, we can see that Makholm's result is a strengthening of the second one, in the sense that it no longer requires \MP{} nor the discreteness assumption.
This becomes possible by the usage of \HA{}-inseparable formulas, overcoming the need for Overspill, which---in the proof of \autoref{tennenbaum-inseparable}---can be identified to be the root for these additional assumptions.
In general, we can observe that the results become progressively stronger and less reliant on further assumptions as more and more intermediary results are proven on the object level of \HA{}. 
For example, instead of using Overspill to establish infinite coding, we later use that the coding lemma holds internally (\autoref{hypothesis_object_level_coding}).
Likewise, using \HA-inseparable formulas, instead of inseparable formulas, contributed to another strengthening. 
This might not be the end of the strengthenings that can be done, as it is remains open whether the constructively stronger \( \M > \Nat \to \exists d. \, \neg \Dec{(\divNat d)} \) can also be shown.

As was pointed out by McCarty in~\cite{McCarty}, a weaker version \WCT{} of \CT{} suffices for his proof, where the code representing a given function is hidden behind a double negation.
He mentions in~\cite{mccarty1991incompleteness} that \WCT{} is still consistent with the Fan theorem, while \CT{} is not.
Analogously, the following weakening of \CTQ{} suffices for all of the proofs that we have presented:
\setCoqFilename{FOL.Tennenbaum.Church}
\begin{definition}[\WCTQ][WCT_Q]
  For every function \( f : \Nat \to \Nat \) there \emph{potentially} is a binary \( \Sigma_1 \)--formula \( \varphi_f(x,y) \) such that for every \( n \col \Nat \) we have \( \Q \vdash \forall y. \, \varphi_f(\num n, y) \iff \num{f n} = y \).
\end{definition}
This only needs few changes of the presented proofs.\footnote{
We could have presented all results with respect to \WCTQ{}. We opted against this in favor of \CTQ{}, to avoid additional handling of double negations and to keep the proofs more readable.
}
An advantage of \WCTQ{} over \CTQ{} is that the former follows from the double negation of the latter and is therefore negative, ensuring that its assumption does not block the reduction of proof-terms found in the mechanization~\cite{coquand2013negative}.

Depending on the fragment of first-order logic, one can give constructive proofs of the model existence theorem~\cite{forster2021completeness}, producing a countable syntactic model with computable functions for every consistent theory.
By the argument given in the introduction, model existence would yield a countable and computable non-standard model of \PA, which at first glance seems to contradict the statement of Tennenbaum's theorem.
For any countable non-standard model of \PA{} however, \autoref{Makholm} and \autoref{decidable_div} entail that neither equality nor apartness can be decidable.
This is similar in spirit to the results in~\cite{tomasz2017computable}, showing that even if the functions of the model are computable, non-computable behavior still emerges, but in relation to equality.

\subsection{Coq Mechanization}

The Coq development is axiom-free and the usage of crucial but constructively justified assumptions \CTQ{}, \MP{} and \UC{} is localized in the relevant sections.
Apart from these, there is \autoref{hypothesis_object_level_coding} which is taken as an additional assumption in the relevant sections.
We have given details as to how this hypothesis can be proven, but since we did not yet mechanize the proof, we wanted to make its assumption on the level of the mechanization very explicit, by labeling it as a hypothesis in the accompanying text.

The development depends on and contributes to a growing collaborative Coq library for first-order logic~\cite{kirst2022library}.
Restricting to the files for this project, the line count is roughly 3600 lines of code.
From those, 1500 lines on basic results about \Nat{} and \PA{} models were reused from earlier work~\cite{kirst2021synthetic}.
Notably, the formalization of the various coding lemmas from Section~\ref{section_coding} took 470 lines and all variants of Tennenbaum's theorem together amount to a total of only 800 lines.

\subsection{Related Work}

Classical proofs of Tennenbaum's theorem can be found in~\cite{boolos2002computability, smith2014tennenbaum, kaye2011tennenbaum}.
There are also refinements of the theorem which show that computability of either operation suffices~\cite{mcaloon1982complexity} or which reduce the argument to a weaker induction scheme~\cite{wilmers1985bounded, cegielski1982modeles}.
Constructive accounts were given by McCarty~\cite{mccarty1987variations, McCarty}
and Plisko~\cite{plisko1990constructive}, and a relatively recent investigation into Tennenbaum phenomena was conducted by Godziszewski and Hamkins~\cite{tomasz2017computable}. 

For an account of \CT{} as an axiom in constructive mathematics we refer to Kreisel~\cite{kreisel1970church} and Troelstra~\cite{troelstra1973metamathematical}.
Investigations into \CT{} and its connections to other axioms of synthetic computability based on constructive type theory were done by Forster~\cite{forster:LIPIcs:2021:13455, forster2022church}.
While there is no proof for the consistency of \CT{} in \CIC{}, there are consistency proofs for very similar systems~\cite{yamada2020game, swan2019churchs, forster2022church}.

Compared to the previous conference paper~\cite{hermes2022tennenbaum}, this extended journal version relies on the slightly different definition of \( \Sigma_1 \) and \( \Delta_1 \)--formulas used by Kirst and Peters~\cite{kirst2023goedel}. 
Moreover, they give a derivation of our formulation of \CTQ{} from a more conventional formulation of Church's thesis, illustrating that \CTQ{} is a convenient axiom for sidestepping much of the first-order encoding overhead, while the work needed to formally capture computation within \Q is feasible in its mechanization.

Presentations of first-order logic in the context of proof-checking have already been discussed and used, among others, by Shankar~\cite{shankar1986proof}, Paulson~\cite{paulson2015mechanized}, O'Connor~\cite{oconnor2005essential}, as well as Han and van Doorn~\cite{han2020formal}.
We make use of the Coq library for first-order logic~\cite{kirst2022library,hostert2021toolbox}, which has evolved from several previous projects~\cite{forster2019synthetic,forster2021completeness,kirst2020trakhtenbrot,kirst2021synthetic} and depends on the Coq library of undecidability proofs~\cite{forster2020coq}.

Synthetic computability theory was introduced by Richman and Bauer~\cite{richman1983church,bauer2006first} and initially applied to constructive type theory by Forster, Kirst, and Smolka~\cite{forster2019synthetic}.
Their synthetic approach to undecidability results has been used in several other projects, all merged into the Coq library of undecidability proofs~\cite{forster2020coq}.

\subsection{Future Work}

By relying on the synthetic approach, our treatment of Tennenbaum's theorem does not explicitly mention the computability of addition or multiplication of the model.
To make these assumptions explicit again, and to also free our development from the necessity of adapting this viewpoint, we could assume an abstract version of \CT{} which makes reference to a \( T \) predicate~\cite{kleene1943recursive, forster:LIPIcs:2021:13455} and is then used to axiomatize \( T \)-computable functions.
Combining this with a version of \CT{} that stipulates the computability of every \( T \)-computable function, would then allow us to specifically assume \( T \)-computability for either addition or multiplication, and to formalize the result that \( T \)-computability of either operation leads to the model being standard~\cite{mcaloon1982complexity}.

In this work, we have mechanized one of the two hypotheses that remained unmechanized in~\cite{hermes2022tennenbaum}, leaving solely the object level coding lemma (\autoref{hypothesis_object_level_coding}) yet to be mechanized. 
Achieving this will necessitate the transformation of the proof provided for \autoref{coding_nat} into a derivation using the axioms of \HA{}. 
This task will require the establishment of standard results about prime numbers, akin to what we established on the level of Coq's logic in one of the files.
Similar to the mechanization of \autoref{lemma_HA_inseparable}, these proofs will significantly benefit from the proof mode developed in~\cite{hostert2021toolbox}.
However, even with the proof mode, the translation of these Coq proofs into derivations entails a substantial amount of engineering work. 
Our preliminary estimates suggest that this will be a relatively independent project on its own, hence we leave it for future exploration.

A more satisfying rendering of McCarty's result would be achieved by changing the semantics (\autoref{def_Tarski_semantics}), and putting the interpretations of formulas on the (proof-relevant) type level instead of the propositional level, therefore removing the need to assume \( \UC \) to break the barrier from the propositional to the type level.

Following usual practice in textbooks, we consider the first-order equality symbol as a syntactic primitive and only regard models interpreting it as actual equality in Coq.
When treated as axiomatized relation instead, we could consider the (slightly harder to work with) setoid models and obtain the more general result that no computable non-standard setoid model exists.

There are interesting parallels when comparing the proofs of Tennenbaum's theorem and proofs of the first incompleteness theorem.
In particular, we saw that the usage of \HA{}-inseparable sets, and therefore the usage of Rosser's trick, leads to an improvement of the constructive Tennenbaum result.
Connections between the two theorems are well-known~\cite{kaye1991models,kaye2011tennenbaum}, but it should be interesting to combine the presented work with work like~\cite{kirst2023goedel}, to study their connection in a constructive framework.
We hope that this work has illustrated the value of such projects, as they can shed new light on old proofs, bringing forward their constructive content.

\section*{Acknowledgments}
\noindent The authors want to thank Yannick Forster, Johannes Hostert, Benjamin Peters, and the anonymous reviewers for their helpful comments and suggestions.

\bibliographystyle{alphaurl}
\bibliography{references}

\appendix
\section{Deduction Systems}\label{appendix_natural_deduction}

Intuitionistic natural deduction \( \vdash ~: \List{\Form} \rightarrow \Form \rightarrow \Prop \) is defined inductively by the rules
\begin{center}
  \(\begin{gathered}
    \infer{\Gamma\vdash \varphi}{\varphi\in \Gamma}\hspace{3em}
    \infer{\Gamma\vdash \varphi}{\Gamma\vdash \bot}\hspace{3em}
    \infer{\Gamma\vdash \varphi\simp\psi}{\Gamma,\varphi\vdash \psi}\hspace{3em}
    \infer{\Gamma\vdash \psi}{\Gamma\vdash \varphi\simp\psi&\Gamma\vdash\varphi}
    \\[0.3cm]
    \infer{\Gamma\vdash \varphi\sand\psi}{\Gamma\vdash \varphi& \Gamma\vdash \psi}\hspace{3em}
    \infer{\Gamma\vdash \varphi}{\Gamma\vdash \varphi\sand\psi}\hspace{3em}
    \infer{\Gamma\vdash \psi}{\Gamma\vdash \varphi\sand\psi}
    \\[0.3cm]
    \infer{\Gamma\vdash \varphi\sor\psi}{\Gamma\vdash \varphi}\hspace{3em}
    \infer{\Gamma\vdash \varphi\sor\psi}{\Gamma\vdash \psi}\hspace{3em}
    \infer{\Gamma\vdash\theta}{\Gamma\vdash\varphi\sor\psi&\Gamma,\varphi\vdash \theta&\Gamma,\psi\vdash \theta}
    \\[0.3cm]
    \infer{\Gamma\vdash\sforall \varphi}{ \Gamma[\uparrow]\vdash \varphi}\hspace{3em}
    \infer{\vphantom{\forall}\Gamma\vdash \varphi[t]}{\Gamma\vdash  \sforall \varphi}\hspace{3em}
    \infer{\Gamma\vdash \sexists\varphi}{\Gamma\vdash \varphi[t]}\hspace{3em}
    \infer{\vphantom{\sforall}\Gamma\vdash \psi}{\Gamma\vdash\sexists\varphi& \Gamma[\uparrow],\varphi\vdash \psi[\uparrow]}
    \end{gathered}\)
\end{center}
where \( \varphi[t] \) for a term \( t \) is a short for \( \varphi[\, t ; \lambda n. \, x_n ] \).
We get the classical variant \( \cvdash \) by adding Peirce's rule as an axiom:
\begin{center}
  \( \infer{\Gamma \vdash_c ((\varphi \to \psi) \to \varphi) \to \varphi }{} \)
\end{center}
The deduction systems lift to possibly infinite contexts \( \MCL T:\Form\to\Prop \) by writing \( \MCL T\vdash \varphi \) if there is a finite \( \Gamma\subseteq \MCL T \) with \( \Gamma\vdash\varphi \).

\end{document}